\documentclass[12pt,twoside,reqno]{amsart}
\usepackage{amsmath}
\usepackage{amsfonts}
\usepackage{amssymb}
\usepackage{color}
\usepackage{mathrsfs}
\usepackage{cite}
\usepackage{geometry}
\usepackage{marginnote}%\marginpar{}, \marginnote{}
\usepackage{todonotes}%\todo{}
\newtheorem{theorem}{Theorem}[section]
\newtheorem{corollary}[theorem]{Corollary}
\newtheorem{lemma}[theorem]{Lemma}
\newtheorem{proposition}[theorem]{Proposition}

\numberwithin{equation}{section}
\begin{document}
\title{Global bounded and unbounded solutions to a chemotaxis system with indirect signal production} 
\thanks{}
\author{Philippe Lauren\c{c}ot}
\address{Institut de Math\'ematiques de Toulouse, UMR~5219, Universit\'e de Toulouse, CNRS \\ F--31062 Toulouse Cedex 9, France}
\email{laurenco@math.univ-toulouse.fr}
\keywords{chemotaxis system - global existence - boundedness - unbounded global solutions}
\subjclass{35B40 - 35M33 - 35K10 - 35Q92}

\date{\today}

%%%%%%%%%%%%%%%%
%%%%%%%%%%%%%%%%
\begin{abstract}
The well-posedness of a chemotaxis system with indirect signal production in a two-dimensional domain is shown, all solutions being global unlike the classical Keller-Segel chemotaxis system. Nevertheless, there is a threshold value $M_c$ of the mass of the first component which separates two different behaviours: solutions are bounded when the mass is below $M_c$ while there are unbounded solutions starting from initial conditions having a mass exceeding $M_c$. This result extends to arbitrary two-dimensional domains a previous result of Tao \& Winkler (2017) obtained for radially symmetric solutions to a simplified version of the model in a ball and relies on a different approach involving a Liapunov functional.
\end{abstract}
%%%%%%%%%%%%%%%%
%%%%%%%%%%%%%%%%

\maketitle

%
%     HEADLINES
%
\pagestyle{myheadings}
\markboth{\sc{Ph.~Lauren\c cot}}{\sc{Global bounded and unbounded solutions to a chemotaxis system}}

%%%%%%%%%%%%%%%%
%%%%%%%%%%%%%%%%
\section{Introduction}\label{s1}
%%%%%%%%%%%%%%%%
%%%%%%%%%%%%%%%%

Chemotaxis models have been derived in \cite{PMW1999, STP2013, WP1998} to describe the spreading and aggregative behaviour of the mountain pine beetle which has a major impact on the forest industry in North America. These models describe the space and time evolution of the density $u$ of flying beetles, the density $v$ of nesting beetles, and the concentration $w$ of beetle pheromone (chemoattractant). The main behavioural difference between flying and nesting beetles is that the latter do not disperse in space while the former move randomly in space, their motion being biased by high gradients of the pheromone concentration. A very important feature in the model is that the pheromone is produced by the nesting beetles and not by the flying ones, though it influences only the motion of the latter. This is in sharp contrast with the classical Keller-Segel model where the chemical inducing a bias in the motion of the species is produced directly by the species itself. As we shall see below, this feature alters significantly the dynamics, at least in two space dimensions. 

A simplified version of the models derived in \cite{PMW1999, STP2013, WP1998} is considered in \cite{TW2017} and reads
\begin{subequations}\label{tw}
\begin{align}
\partial_t u & = \mathrm{div}\left( \nabla u - u \nabla w \right) \;\;\text{ in }\;\; (0,\infty)\times \Omega\ , \label{tw1}\\ 
\nu\varepsilon \partial_t v & =  u - v \;\;\text{ in }\;\; (0,\infty)\times \Omega\ , \label{tw2}\\
0 & = D \Delta w - \langle v \rangle + v \;\;\text{ in }\;\; (0,\infty)\times \Omega\ , \label{tw3}
\end{align}
supplemented with no-flux boundary conditions for $u$ and $w$ 
\begin{equation}
\nabla u \cdot \mathbf{n} = \nabla w \cdot \mathbf{n} = 0 \;\;\text{ on }\;\;  (0,\infty)\times\partial\Omega\ , \label{tw4}
\end{equation}
the normalization condition for $w$
\begin{equation}
\langle w \rangle = 0 \;\;\text{ in }\;\;  (0,\infty)\ , \label{tw5}
\end{equation}
\end{subequations} 
and non-negative initial conditions $(u^{in},v^{in})$ for $(u,v)$. Here $\Omega$ is a smooth bounded domain of $\mathbb{R}^2$, $\nu$, $\varepsilon$, and $D$ are positive parameters, and $\langle v \rangle$ in \eqref{tw3} and $\langle w \rangle$ in \eqref{tw5} denote the mean value with respect to the space variable of $v$ and $w$, respectively. Recall that, for $z\in L^1(\Omega)$, its mean value is given by
\begin{equation*}
\langle z \rangle := \frac{1}{|\Omega|} \int_\Omega z(x)\ \mathrm{d}x\ .
\end{equation*} 
A J\"ager-Luckhaus version of the Keller-Segel chemotaxis system \cite{JaLu1992} is recovered from \eqref{tw} by setting $\nu\varepsilon=0$ and reads, since $v=u$ in that case,
\begin{subequations}\label{jl}
\begin{align}
\partial_t u & = \mathrm{div}\left( \nabla u - u \nabla w \right) \;\;\text{ in }\;\; (0,\infty)\times \Omega\ , \label{jl1}\\ 
0 & = D \Delta w - \langle u \rangle + u \;\;\text{ in }\;\; (0,\infty)\times \Omega\ , \label{jl2}
\end{align}
supplemented with no-flux boundary conditions for $u$ and $w$
\begin{equation}
\nabla u \cdot \mathbf{n} = \nabla w \cdot \mathbf{n} = 0 \;\;\text{ on }\;\;  (0,\infty)\times\partial\Omega\ , \label{jl3}
\end{equation}
the normalization condition for $w$
\begin{equation}
\langle w \rangle = 0 \;\;\text{ in }\;\;  (0,\infty)\ , \label{jl4}
\end{equation}
\end{subequations} 
and a non-negative initial condition $u^{in}$ for $u$. Observe that, in \eqref{jl}, the chemoattractant is produced directly by the species with density $u$, instead of being produced through another species as in \eqref{tw}. A first consequence of this feature is that all solutions to \eqref{tw} are global \cite[Theorem~1.1]{TW2017}. This global existence property contrasts markedly with the situation for \eqref{jl} for which the following is known: solutions to \eqref{jl} are global when either $\|u^{in}\|_1<4\pi D$ or $\Omega$ is a ball, $u^{in}$ is radially symmetric, and $\|u^{in}\|_1>8\pi D$ \cite{Bi1998, JaLu1992, Na1995}. Finite time blowup may take place when either $\|u^{in}\|_1>4\pi D$ or $\Omega$ is a ball, $u^{in}$ is radially symmetric, and $\|u^{in}\|_1>8\pi D$ \cite{JaLu1992, Na1995, Na2001, SeSu2003}, see also the survey \cite{Ho2003}.

A striking feature of the dynamics of \eqref{tw}, uncovered in \cite{TW2017}, is that, though all solutions to \eqref{tw} are global, a threshold phenomenon occurs in infinite time. More precisely, when $\Omega$ is a ball, say $\Omega=B_1(0) := \{ x\in \mathbb{R}^2\ :\ |x|<1\}$, and the initial conditions $(u^{in},v^{in})$ are radially symmetric, so that the solution $(u,v,w)$ of \eqref{tw} is also radially symmetric for all times, it is shown in \cite[Theorems~1.2 \&~1.3]{TW2017} that all (radially symmetric) solutions are actually bounded when $\|u^{in}\|_1<8\pi D$ while there are unbounded solutions emanating from initial conditions satisfying $\|u^{in}\|_1>8\pi D$, the $L^\infty$-norm of $u(t)$ growing up as $t\to\infty$ at an exponential rate. The approach of \cite{TW2017} exploits the classical fact that the specific structure of \eqref{tw} and the radial symmetry of the solutions allow one to reduce \eqref{tw} to a single nonlocal parabolic equation for the cumulative distribution function $U$ defined by 
\begin{equation*}
U(t,|x|) := \int_0^{\sqrt{|x|}} \sigma \bar{u}(t,\sigma)\ \mathrm{d}\sigma\ , \qquad (t,x)\in (0,\infty)\times B_1(0)\ ,
\end{equation*}
where $\bar{u}(t,|x|):=u(t,x)$ for $(t,x)\in (0,\infty)\times B_1(0)$. Even though the parabolic equation solved by $U$ involves a nonlocal term, it turns out that such a powerful tool as the comparison principle becomes available, see \cite[Lemma~4.2]{TW2017}, and the proofs of \cite[Theorems~1.2 \&~1.3]{TW2017} rely on the construction of suitable supersolutions and subsolutions. 

This technique is however restricted to \eqref{tw} in a radially symmetric setting and does not allow one to handle more general initial conditions or to consider an arbitrary domain $\Omega\subset \mathbb{R}^2$. Neither does it extend to the version of \eqref{tw} involving degradation of the chemoattractant
\begin{subequations}\label{twd}
\begin{align}
\partial_t u & = \mathrm{div}\left( \nabla u - u \nabla w \right) \;\;\text{ in }\;\; (0,\infty)\times \Omega\ , \label{twda}\\ 
\nu\varepsilon \partial_t v & =  u - v \;\;\text{ in }\;\; (0,\infty)\times \Omega\ , \label{twdb}\\
0 & = D \Delta w - \delta w + v \;\;\text{ in }\;\; (0,\infty)\times \Omega\ , \label{twdc}
\end{align}
\end{subequations} 
with $\delta>0$, supplemented with no-flux boundary conditions for $u$ and $w$ and non-negative initial conditions $(u^{in},v^{in})$ for $(u,v)$, nor to its parabolic counterpart
\begin{subequations}\label{a1}
\begin{align}
\partial_t u & = \mathrm{div}\left( \nabla u - u \nabla w \right) \;\;\text{ in }\;\; (0,\infty)\times \Omega\ , \label{a1a}\\ 
\nu\varepsilon \partial_t v & =  u - v \;\;\text{ in }\;\; (0,\infty)\times \Omega\ , \label{a1b}\\
\nu \partial_t w & = D \Delta w - \delta w + v \;\;\text{ in }\;\; (0,\infty)\times \Omega\ , \label{a1c}
\end{align}
supplemented with no-flux boundary conditions for $u$ and $w$
\begin{equation}
\left( \nabla u - u \nabla w \right)\cdot \mathbf{n} = \nabla w \cdot \mathbf{n} = 0 \;\;\text{ on }\;\;  (0,\infty)\times\partial\Omega\ , \label{a1d} 
\end{equation}
\end{subequations}
and initial conditions
\begin{equation}
(u,v,w)(0) = (u^{in}, v^{in},w^{in}) \;\;\text{ in }\;\; \Omega\ . \label{a2}
\end{equation}

The main contribution of this paper is to show that the systems \eqref{tw}, \eqref{twd}, and \eqref{a1} share the same infinite time threshold phenomenon uncovered in \cite{TW2017}, without restrictions on the two-dimensional domain $\Omega$ and the initial data. This requires a different argument and we actually show that \eqref{tw}, \eqref{twd}, and \eqref{a1} all possess a Liapunov functional, and that the properties of this Liapunov functional provide insight on the boundedness or unboundedness of the solutions. 

From now on, we shall focus on \eqref{a1}-\eqref{a2} and will return briefly to \eqref{tw} and \eqref{twd} in Section~\ref{s5}. Introducing 
\begin{equation}
L(r) := r \ln{r} - r + 1 \ge 0\ , \qquad r\ge 0\ , \label{a4}
\end{equation}
and
\begin{subequations}\label{a3}
\begin{align}
\mathcal{E}_0(u,w) & := \int_\Omega \left[ L(u) - uw \right]\ \mathrm{d}x + \frac{D}{2} \|\nabla w\|_2^2 + \frac{\delta}{2} \|w\|_2^2\ , \label{a3a} \\ 
\mathcal{E}(u,v,w) & := \mathcal{E}_0(u,w) + \frac{\varepsilon}{2} \| -D \Delta w + \delta w - v \|_2^2\ , \label{a3b}
\end{align}
\end{subequations}
the main observation is that $\mathcal{E}$ is a Liapunov functional for \eqref{a1}-\eqref{a2}, see Lemma~\ref{lemd2}. Interestingly, $\mathcal{E}_0$ is a Liapunov functional for the classical Keller-Segel system \cite{GZ1998, NSY1997}
\begin{subequations}\label{ks}
\begin{align}
\partial_t u & = \mathrm{div}\left( \nabla u - u \nabla w \right) \;\;\text{ in }\;\; (0,\infty)\times \Omega\ , \label{ks1}\\ 
\nu \partial_t w & = D \Delta w - \delta w + u \;\;\text{ in }\;\; (0,\infty)\times \Omega\ , \label{ks2} \\
& \nabla u \cdot \mathbf{n} = \nabla w \cdot \mathbf{n} = 0 \;\;\text{ on }\;\;  (0,\infty)\times\partial\Omega\ , \label{ks3} \\
& (u,w)(0) = (u^{in},w^{in}) \;\;\text{ in }\;\; \Omega\ , \label{ks4}
\end{align}
\end{subequations} 
and, as such, its properties have been thoroughly studied \cite{GZ1998, Ho2001, Ho2002, HW2001, NSY1997}. In particular, introducing the set 
\begin{equation}
\mathcal{A}_M := \left\{ (u,w)\in L_1(\Omega)\times W_2^1(\Omega)\ :\ u\ge 0\ , \ w\ge 0\ , \ \|u\|_1=M \right\} \label{a9}
\end{equation}
for $M\ge 0$, it is known that 
\begin{equation}
\begin{array}{lcl}
\displaystyle{\inf_{(u,w)\in \mathcal{A}_M} \mathcal{E}_0(u,w)} > - \infty  & \text{ if } & M\in [0,4\pi D] \ , \\
\displaystyle{\inf_{(u,w)\in \mathcal{A}_M} \mathcal{E}_0(u,w)} = - \infty  & \text{ if } & M >4\pi D \ , 
\end{array} \label{a10}
\end{equation}
a property which has far-reaching consequences on the dynamics of \eqref{ks}. Indeed, given $M>0$, global existence and blowup of solutions to \eqref{ks} are intimately related to the finiteness or not of the infimum of $\mathcal{E}_0$ on $\mathcal{A}_M$  \cite{GZ1998, Ho2001, Ho2002, HW2001, NSY1997}. As we shall see below, the property \eqref{a10} also plays an important role in the dynamics of \eqref{a1}-\eqref{a2}, a feature which is actually not so surprising as there is a close connection between \eqref{a1}-\eqref{a2} and \eqref{ks}. Indeed, \eqref{ks} can formally be derived from \eqref{a1}-\eqref{a2} by setting $\varepsilon=0$. 

We now describe our results on \eqref{a1}-\eqref{a2} and begin with its well-posedness. For $\theta\in (2/3,1)$ and $M\ge 0$, we set
\begin{equation}
\mathcal{I}_{M,\theta} := \left\{ (u,v,w)\in W_{3,\mathcal{B},+}^{2\theta}(\Omega)\times W_{3,\mathcal{B},+}^1(\Omega)\times W_{3,\mathcal{B},+}^2(\Omega)\ :\ \ \|u\|_1=M  \right\}\ , \label{d20}
\end{equation}
where
\begin{equation}
W_{r,\mathcal{B}}^{m}(\Omega) := \left\{ 
\begin{array}{ll}
\left\{ z \in W^{m}_{r}(\Omega)\ :\ \nabla z\cdot \mathbf{n} = 0 \;\text{ on }\; \partial\Omega \right\}\ , & 1+1/r < m \le 2\ , \\
& \\
W^{m}_{r}(\Omega) \ , & -1+1/r < m < 1+1/r \ , \\
& \\
W^{-s}_{r/(r-1)}(\Omega)' \ , & -2+1/r < s \le -1+1/r \ , 
\end{array}\right. \label{WmrB}
\end{equation} 
and 
\begin{equation}
W_{r,\mathcal{B},+}^{m}(\Omega) := \left\{ z\in W_{r,\mathcal{B}}^{m}(\Omega)\ :\ z \ge 0 \;\text{ in }\; \Omega\right\} \label{WmrBp}
\end{equation} 
for $r\in (1,\infty)$, see \cite[Section~5]{Am1993}. 

%%%%%%%%%%%%%%%%
\begin{theorem}\label{thma1}
Let $\theta\in (5/6,1)$ and consider initial conditions $(u^{in},v^{in},w^{in})\in \mathcal{I}_{M,\theta}$. Then the system \eqref{a1}-\eqref{a2} has a unique non-negative weak solution $(u,v,w)$ in $W^{1}_{3}$ defined on $[0,\infty)$ satisfying
\begin{align*}
u & \in C([0,\infty);W_{3,\mathcal{B}}^{2\theta}(\Omega)) \cap C^1([0,\infty);W^{2\theta-2}_{3,\mathcal{B}}(\Omega)) \ , \\
v & \in C^1([0,\infty);W^{1}_{3}(\Omega))\ , \\
w & \in C([0,\infty);W_{3,\mathcal{B}}^{2}(\Omega)) \cap C^1([0,\infty);L_3(\Omega)) \ ,
\end{align*}
and
\begin{equation}
\|u(t)\|_1 = M := \|u^{in}\|_1\ , \qquad t\ge 0\ . \label{a6}
\end{equation}
In particular, $(u(t),v(t),w(t))\in \mathcal{I}_{M,\theta}$ for all $t\ge 0$. Furthermore,
\begin{equation*}
u \in C((0,\infty);W_{3,\mathcal{B}}^{2}(\Omega)) \cap C^1((0,\infty);L_{3}(\Omega)) \ .
\end{equation*}
\end{theorem}
%%%%%%%%%%%%%%%%

Several approaches may be used to deal with the well-posedness of \eqref{a1}-\eqref{a2}. Since \eqref{a1b} can be solved explicitly to find $v$ in terms of $u$, one possibility would be to solve the parabolic system \eqref{a1a}-\eqref{a1c}-\eqref{a1d}-\eqref{a2} with a source term which is nonlocal with respect to time. However, since an abstract theory is developed in \cite{Am1991} to handle systems coupling parabolic equations and ordinary differential equations, we rather follow this route which provides the local well-posedness of \eqref{a1}-\eqref{a2} in $W_3^1(\Omega)$ and may cope as well with nonlinear reaction terms in \eqref{a1b} or \eqref{a1c}. The parabolic regularising effects of both \eqref{a1a} and \eqref{a1c} are then used to prove that it is actually a strong solution for positive times. Several estimates are next needed to prove that the solution is global, the positivity of $\varepsilon$ being of utmost importance already in Lemma~\ref{lemc1}. 

Having established the global well-posedness of \eqref{a1}-\eqref{a2}, we next turn to qualitative information on the dynamics of \eqref{a1}-\eqref{a2} which is the main goal of this paper. We first study the boundedness of solutions, thereby extending \cite[Theorem~1.2]{TW2017} to \eqref{a1}-\eqref{a2} in an arbitrary domain $\Omega\subset \mathbb{R}^2$.

%%%%%%%%%%%%%%%%
\begin{theorem}\label{thma2}
Let $\theta\in (5/6,1)$ and consider initial conditions $(u^{in},v^{in},w^{in})\in \mathcal{I}_{M,\theta}$. We denote the corresponding solution to \eqref{a1}-\eqref{a2} given by Theorem~\ref{thma1} by $(u,v,w)$.
\begin{itemize}
	\item[(a)] Assume further that $M= \|u^{in}\|_1 \in (0,4\pi D)$. Then
	\begin{equation}
	\sup_{t\ge 0} \left\{ \|u(t)\|_\infty + \|v(t)\|_\infty + \|w(t)\|_{W_\infty^1} \right\} < \infty\ . \label{a7}
	\end{equation}
	\item[(b)] Assume further that $\Omega$ is the ball $B_r(0)$ of radius $r>0$ centered at $x=0$, $(u^{in},v^{in},w^{in})$ are radially symmetric, and $M= \|u^{in}\|_1 \in (0,8\pi D)$. Then \eqref{a7} also holds true.
\end{itemize} 
\end{theorem}
%%%%%%%%%%%%%%%%

The proof of Theorem~\ref{thma2}~(a) relies on the properties of the Liapunov functional $\mathcal{E}$ (defined in \eqref{a3b}) which are derived from the properties of $\mathcal{E}_0$ (defined in \eqref{a3a}) \cite{GZ1998, NSY1997}. The main observation is that, when $(u,w)$ belongs to the set $\mathcal{A}_M$ defined in \eqref{a9} and $M\in (0,4\pi D)$, it follows from the Moser-Trudinger inequality \cite{CY1988} that 
\begin{equation*}
\mathcal{E}_0(u,w)\ge \frac{4\pi D - M}{8\pi} \|\nabla w\|_2^2 + \frac{\delta}{2} \|w\|_2^2 - b
\end{equation*}
for some constant $b\ge 0$ which depends only on $\Omega$, $D$, $\delta$, $M$, and $\|w\|_1$. Therefore, an upper bound on $\mathcal{E}_0(u,w)$ entails an estimate on the $W_2^1$-norm of $w$ which, in turn, allows us to derive further estimates and deduce the boundedness of the solution. Similar arguments are used to prove Theorem~\ref{thma2}~(b).

The final result deals with the unboundedness of some solutions provided $\|u^{in}\|_1$ is large enough.

%%%%%%%%%%%%%%%%
\begin{theorem}\label{thma3}
Let $\theta\in (5/6,1)$.
\begin{itemize}
		\item [(a)] Given $M\in (4\pi D,\infty) \setminus 4\pi D \mathbb{N}$, there are non-negative functions $(u^{in},v^{in},w^{in})\in \mathcal{I}_{M,\theta}$ for which the corresponding solution $(u,v,w)$ to \eqref{a1}-\eqref{a2} given by Theorem~\ref{thma1} satisfies
		\begin{equation}
		\limsup_{t\to\infty} \|u(t)\|_\infty = \infty\ . \label{a8}
		\end{equation}
		\item [(b)] Assume further that $\Omega$ is the ball $B_r(0)$ of radius $r>0$ centered at $x=0$. Given $M\in (8\pi D,\infty)$, there are non-negative radially symmetric functions $(u^{in},v^{in},w^{in})\in \mathcal{I}_{M,\theta}$ for which the corresponding solution $(u,v,w)$ to \eqref{a1}-\eqref{a2} given by Theorem~\ref{thma1} satisfies \eqref{a8}.
	\end{itemize}
\end{theorem}
%%%%%%%%%%%%%%%%

 On the one hand, Theorem~\ref{thma3} extends \cite[Theorem~1.3]{TW2017} to \eqref{a1}-\eqref{a2} in an arbitrary domain $\Omega\subset \mathbb{R}^2$ and, when $\Omega$ is a ball and the initial data are radially symmetric, applying the same approach to \eqref{tw} is likely to provide an alternative proof of \cite[Theorem~1.3]{TW2017}. It however only gives the unboundedness of the solution, while \cite[Theorem~1.3]{TW2017} guarantees an exponential growth of the $L_\infty$-norm of $u$. On the other hand, Theorem~\ref{thma3} is also a consequence of the availability of a Liapunov functional and its properties, the latter being established in \cite{Ho2001, Ho2002, HW2001}. Roughly speaking, the proof of Theorem~\ref{thma3}  involves three steps and proceeds along the lines of \cite{HW2001}. We first show that, given a solution $(u,v,w)$ to \eqref{a1}-\eqref{a2} such that $u\in L_\infty((0,\infty)\times\Omega)$, then there are a sequence $(t_k)_{k\ge 1}$, $t_k\to\infty$, of positive real numbers and a stationary solution $(u_*,v_*,w_*)$ to \eqref{a1} such that
 \begin{equation}
 \mathcal{E}(u_*,v_*,w_*) \le \liminf_{k\to\infty} \mathcal{E}(u(t_k),v(t_k),w(t_k))\ . \label{a11}
 \end{equation}
 It next follows from \cite[Lemma~3.5]{HW2001} that there is $\mu_M\ge 0$ depending only on the parameters in \eqref{a1} and $M$ such that
 \begin{equation}
 \mathcal{E}(u_*,v_*,w_*) \ge -\mu_M\ . \label{a12}
 \end{equation}
 It now readily follows from \eqref{a11}, \eqref{a12}, and the fact that $\mathcal{E}$ is a Liapunov functional that, if $\mathcal{E}(u^{in},v^{in},w^{in})<-\mu_M$, then $(u,v,w)$ cannot be bounded. The last step is to check that such initial conditions exist, but this is again a consequence of the analysis performed in \cite{Ho2001, Ho2002, HW2001}.
 
%%%%%%%%%%%%%%%%
%%%%%%%%%%%%%%%%
\section{Global existence}\label{s2}
%%%%%%%%%%%%%%%%
%%%%%%%%%%%%%%%%

%%%%%%%%%%%%%%%%
%%%%%%%%%%%%%%%%
\subsection{Local existence}\label{s2.1}
%%%%%%%%%%%%%%%%
%%%%%%%%%%%%%%%%

We begin with the local existence of solutions to \eqref{a1}-\eqref{a2} which is a consequence of the well-posedness theory developped in	\cite{Am1991} for partially diffusive systems. 

%%%%%%%%%%%%%%%%
\begin{proposition}\label{propb1}
Given non-negative initial conditions $(u^{in},v^{in},w^{in})$ in $W^{1}_{3}(\Omega;\mathbb{R}^3)$, the system \eqref{a1}-\eqref{a2} has a unique non-negative weak solution $(u,v,w)$ in $W^{1}_{3}$ defined on a maximal time interval $[0,T_m)$, $T_m\in (0,\infty]$, satisfying
\begin{align*}
(u,w) & \in C\left( [0,T_m);W^{1}_{3}(\Omega;\mathbb{R}^2) \right)\cap C^1\left( [0,T_m);W^{1}_{3/2}(\Omega;\mathbb{R}^2)' \right)\ , \\
v & \in C^1([0,T_m);W^{1}_{3}(\Omega))\ ,
\end{align*}
and
\begin{equation}
\|u(t)\|_1 = M := \|u^{in}\|_1\ , \qquad t\in [0,T_m)\ . \label{b1}
\end{equation}
Moreover, if there is $T>0$ such that
\begin{equation}
(u,v,w)\in BUC\left( [0,T]\cap [0,T_m); W^{1}_{3}(\Omega;\mathbb{R}^3) \right)\ , \label{b2}
\end{equation}
then $T_m\ge T$. 
\end{proposition}
%%%%%%%%%%%%%%%%

\begin{proof}
In order to cast \eqref{a1}-\eqref{a2} in a form suitable to apply the abstract theory developed in \cite{Am1991}, we set $u_1=u$, $u_2=w$, $u_3=v$, $U:=(u_1,u_2,u_3)$, $U^1:=(u_1,u_2)$, and $U^2:=(u_3)$, and define matrix-valued functions $(a_{j,k}^1)_{1\le j,k\le 2}$, $(a_{j,k}^2)_{1\le j,k\le 2}$ in $\mathcal{M}_2(\mathbb{R})$ and vector-valued functions $(a_j^1)_{1\le j \le 2}$, $(a_j^2)_{1\le j \le 2}$ in $\mathbb{R}^2$ by
\begin{align*}
a_{1,1}^1(U) = a_{2,2}^1(U) := A(U)\ , \qquad a_{1,2}^1(U) = a_{2,1}^1(U) := 0\ , \qquad U\in\mathbb{R}^3\ ,\\
a_{j,k}^2(U) = 0\ , \qquad a_j^1(U) = a_j^2(U) := 0\ , \qquad 1\le j,k\le 2\ , \ U\in\mathbb{R}^3\ ,
\end{align*} 
the $2\times 2$-matrix $A(U)$ being given by
\begin{equation*}
A(U) := \begin{pmatrix}
1 & -u_1 \\
0 & D/\nu
\end{pmatrix}
\end{equation*}
For $V=(v_1,v_2,v_3) \in\mathbb{R}^3$, we define the operators $\mathcal{A}^1(V)$, $\mathcal{A}^2(V)$ and the boundary operators $\mathcal{B}^1(V)$, $\mathcal{B}^2(V)$ by
\begin{align*}
\mathcal{A}^r(V)U^r & := - \sum_{j,k=1}^2 \partial_j \left( a_{j,k}^r(V) \partial_k U^r \right) + \sum_{j=1}^2 a_j^r(V) \partial_j U^r\ , \qquad r=1,2\ , \\
\mathcal{B}^r(V)U^r & := \sum_{j,k=1}^2 a_{j,k}^r(V) \mathbf{n}_j \partial_k U^r \ , \qquad r=1,2\ ,
\end{align*}
as well as the source terms
\begin{equation*}
S^1(V) := \begin{pmatrix}
0 \\
\displaystyle{\frac{v_3-v_2}{\nu}}
\end{pmatrix} \ , \qquad S^2(V) := \left( \frac{v_1-v_3}{\nu\varepsilon}\right)\ .
\end{equation*}
With this notation, the system \eqref{a1}-\eqref{a2} reads 
\begin{subequations}\label{b3}
\begin{align}
\partial_t U^1 + \mathcal{A}^1(U)U^1 + \mathcal{A}^2(U) U^2 & = S^1(U) \;\;\text{ in }\;\; (0,\infty)\times\Omega\ , \label{b3a}\\
\partial_t U^2 & = S^2(U) \;\;\text{ in }\;\; (0,\infty)\times\Omega\ , \label{b3b}\\
\mathcal{B}^1(U)U^1 + \mathcal{B}^2(U)U^2 & = 0 \;\;\text{ on }\;\; (0,\infty)\times\Omega\ , \label{b3c} \\
U(0) & = U^{in} \;\;\text{ in }\;\; \Omega\ , \label{b3d}
\end{align}
\end{subequations}
where $U^{in} := (u^{in}, w^{in}, v^{in}) \in W^{1}_{3}(\Omega;\mathbb{R}^3)$. Since $A(U)$ has two positive eigenvalues $(1,D/\nu)$ for all $U\in\mathbb{R}^3$, it follows from the structure of $(a_{j,k}^1)_{1\le j,k\le 2}$ and \cite[Remarks~4.1~(a)-(iii)]{Am1991} that $(\mathcal{A}^1,\mathcal{B}^1)$ is normally elliptic and, since $\mathcal{B}^2=0$, the condition~(6.1) in \cite{Am1991} is satisfied. We then infer from \cite[Theorem~6.4]{Am1991} that there is a unique weak solution $U=(U^1,U^2)$ to \eqref{b3} defined on a maximal time interval $[0,T_m)$, $T_m\in (0,\infty]$, which satisfies
\begin{align*}
U^1 & \in C\left( [0,T_m); W^{1}_{3}(\Omega;\mathbb{R}^2) \right) \cap C^1\left( [0,T_m); W^{1}_{3/2}(\Omega;\mathbb{R}^2)' \right)\ , \\ 
U^2 & \in C^1\left( [0,T_m); W^{1}_{3}(\Omega) \right)\ .
\end{align*}
Setting $(u,w,v):=U$, we have thus constructed a unique weak solution $(u,v,w)$ to \eqref{a1}-\eqref{a2} endowed with the properties listed in Proposition~\ref{propb1}, except for the non-negativity of $(u,v,w)$ and \eqref{b1}. To prove the former, we first notice that \eqref{a1a}-\eqref{a1d} and the comparison principle entail that $u\ge 0$ in $[0,T_m)\times \Omega$. Owing to \eqref{a1b}, this in turn implies that $v\ge 0$ in $[0,T_m)\times \Omega$. This last property, along with \eqref{a1c}-\eqref{a1d} and a further use of the comparison principle, ensures that $w\ge 0$ in $[0,T_m)\times \Omega$. We finally integrate \eqref{a1a} over $\Omega$ and deduce \eqref{b1} from the no-flux boundary conditions \eqref{a1d} and the non-negativity of $u$, thereby completing the proof.
\end{proof}

We next derive additional regularity properties of $u$ and $w$ with the help of parabolic regularity results.

%%%%%%%%%%%%%%%%
\begin{corollary}\label{corb2}
Consider non-negative initial conditions $(u^{in},v^{in},w^{in})$ in $W^1_3(\Omega;\mathbb{R}^3)$ and denote the corresponding weak solution in $W^{1}_{3}$ to \eqref{a1}-\eqref{a2} by $(u,v,w)$, defined up to its maximal time $T_m\in (0,T_m]$, see Proposition~\ref{propb1}. Assume further that 
\begin{equation}
u^{in}\in W_{3,\mathcal{B}}^{2\theta}(\Omega) \;\;\text{ and }\;\; w^{in} \in W^{2}_{3,\mathcal{B}}(\Omega) \label{b4}
\end{equation}
for some $\theta\in (5/6,1)$. Then 
\begin{align}
u & \in C([0,T_m);W_{3,\mathcal{B}}^{2\theta}(\Omega)) \cap C^1([0,T_m);W^{2\theta-2}_{3,\mathcal{B}}(\Omega)) \ , \label{b5a} \\
u & \in C((0,T_m);W_{3,\mathcal{B}}^{2}(\Omega)) \cap C^1((0,T_m);L_{3}(\Omega)) \ , \label{b5b} \\
w & \in C([0,T_m);W_{3,\mathcal{B}}^{2}(\Omega)) \cap C^1([0,T_m);L_3(\Omega)) \ . \label{b5c}
\end{align}
\end{corollary}
%%%%%%%%%%%%%%%%

\begin{proof}
Let $\beta\in (0,1)$. Since $v\in C^1([0,T_m);W^{1}_{3}(\Omega))$ and $-D \Delta + \delta\, \mathrm{id}$ generates an analytic semigroup in $L_3(\Omega)$, we infer from \eqref{a1c}, \eqref{b4}, and \cite[Theorem~10.1]{Am1993} (with $\rho=\beta$, $E_0=L_3(\Omega)$, and $E_1=W^{2}_{3,\mathcal{B}}(\Omega)$) that 
\begin{align}
w \in C([0,T_m);W_{3,\mathcal{B}}^{2}(\Omega)) \cap C^1([0,T_m);L_3(\Omega)) \ , \label{b6a} \\
w \in C^\beta((0,T_m);W^{2}_{3,\mathcal{B}}(\Omega)) \cap C^{1+\beta}((0,T_m);L_3(\Omega)) \ . \label{b6b}
\end{align}
Next, since
\begin{equation*}
u\in C([0,T_m);W_{3,\mathcal{B}}^1(\Omega)) \cap C^1([0,T_m);W_{3,\mathcal{B}}^{-1}(\Omega))
\end{equation*}
by \eqref{WmrB} and Proposition~\ref{propb1} and 
\begin{equation*}
\left( W_{3,\mathcal{B}}^{-1}(\Omega) , W_{3,\mathcal{B}}^{1}(\Omega) \right)_\theta \stackrel{\cdot}{=} W_{3,\mathcal{B}}^{2\theta-1}(\Omega)
\end{equation*}
(up to equivalent norms) by \cite[Theorem~7.2]{Am1993}, an interpolation argument guarantees that
\begin{equation*}
u \in C^{1-\theta}([0,T_m);W_{3,\mathcal{B}}^{2\theta-1}(\Omega))\ .
\end{equation*}
Owing to \eqref{b6a}, a similar argument implies that $w\in C^{1-\theta}([0,T_m);W_{3,\mathcal{B}}^{2\theta}(\Omega))$. Consequently, recalling that $W_{3,\mathcal{B}}^{2\theta-1}(\Omega)$ is an algebra as $2\theta-1>2/3$, we conclude that $u \nabla w\in C^{1-\theta}([0,T_m);W_{3,\mathcal{B}}^{2\theta-1}(\Omega;\mathbb{R}^2))$. Thus,
\begin{equation}
\mathrm{div}(u\nabla w) \in C^{1-\theta}([0,T_m);W_{3,\mathcal{B}}^{2\theta-2}(\Omega))\ . \label{b7}
\end{equation}
By \cite[Theorem~8.5]{Am1993}, the realization in $W_{3,\mathcal{B}}^{2\theta-2}(\Omega)$ of the Laplace operator with homogeneous Neumann boundary conditions generates an analytic semigroup in $W_{3,\mathcal{B}}^{2\theta-2}(\Omega)$ (with domain $W_{3,\mathcal{B}}^{2\theta}(\Omega)$) and we infer from \eqref{a1a}, \eqref{b4}, \eqref{b7}, and \cite[Theorem~10.1]{Am1993} (with $\rho=1-\theta$, $E_0=W_{3,\mathcal{B}}^{2\theta-2}(\Omega)$, and $E_1=W_{3,\mathcal{B}}^{2\theta}(\Omega)$) that
\begin{equation}
u \in C([0,T_m);W_{3,\mathcal{B}}^{2\theta}(\Omega)) \cap C^1([0,T_m);W_{3,\mathcal{B}}^{2\theta-2}(\Omega))\ . \label{b8}
\end{equation}
Consider next $\alpha\in (5/3-\theta,1)$. Using once more \cite[Theorem~7.2]{Am1993} which guarantees that
\begin{equation*}
\left( W_{3,\mathcal{B}}^{2\theta-2}(\Omega) , W_{3,\mathcal{B}}^{2\theta}(\Omega) \right)_\alpha \stackrel{\cdot}{=} W_{3,\mathcal{B}}^{2(\alpha+\theta)-2}(\Omega)
\end{equation*}
(up to equivalent norms) and an interpolation argument, we deduce from \eqref{b8} that 
\begin{equation*}
u \in C^{1-\alpha}([0,T_m); W_{3,\mathcal{B}}^{2(\alpha+\theta)-2}(\Omega))\ ;
\end{equation*}
Hence, since $2(\alpha+\theta)-2>4/3>1$, 
\begin{equation}
u \in C^{1-\alpha}([0,T_m); W_{3,\mathcal{B}}^{1}(\Omega))\ . \label{b9}
\end{equation}
Combining \eqref{b6b} (with $\beta=1-\alpha$) and \eqref{b9} with the fact that $W_{3,\mathcal{B}}^1(\Omega)$ is an algebra entail that $u\nabla w$ belongs to $C^{1-\alpha}((0,T_m);W_{3,\mathcal{B}}^1(\Omega;\mathbb{R}^2))$ and thus
\begin{equation}
\mathrm{div}(u\nabla w) \in C^{1-\alpha}((0,T_m);L_3(\Omega))\ . \label{b10}
\end{equation}
Owing to \eqref{a1a} and \eqref{b10}, we are again in a position to apply \cite[Theorem~10.1]{Am1993} (with $\rho=1-\alpha$, $E_0=L_3(\Omega)$, $E_1=W_{3,\mathcal{B}}^2(\Omega)$) on any subinterval $[t_0,T_m)$ of $(0,T_m)$ to conclude that 
\begin{equation*}
u \in C^{1-\alpha}((t_0,T_m);W_{3,\mathcal{B}}^2(\Omega))\cap C^{2-\alpha}((t_0,T_m);L_3(\Omega))\ .
\end{equation*}
Since $t_0\in (0,T_m)$ is arbitrary, the proof of Corollary~\ref{corb2} is complete.
\end{proof}
	
%%%%%%%%%%%%%%%%
%%%%%%%%%%%%%%%%
\subsection{Estimates in $W^{1}_{3}(\Omega)$}\label{s2.2}
%%%%%%%%%%%%%%%%
%%%%%%%%%%%%%%%%

\newcounter{NumConstC}

Let $\theta\in (5/6,1)$ and consider initial conditions $(u^{in},v^{in},w^{in})\in \mathcal{I}_{M,\theta}$. We denote the corresponding weak solution in $W^{1,3}$ to \eqref{a1}-\eqref{a2} by $(u,v,w)$, defined up to its maximal time $T_m\in (0,\infty]$, see Proposition~\ref{propb1}.

The aim of this section is to show that \eqref{b2} holds true for any $T>0$. We first recall that, according to \eqref{b1},
\begin{equation}
\|u(t)\|_1 = M = \|u^{in}\|_1\ , \qquad t\in [0,T_m)\ . \label{c1}
\end{equation}
Throughout this section, $C$ and $(C_i)_{i\ge 1}$ are positive constants depending only on $\Omega$, $\nu$, $\varepsilon$, $D$, $\delta$, $\theta$, $u^{in}$, $v^{in}$, and $w^{in}$. Dependence upon additional parameters will be indicated explicitly.

%%%%%%%%%%%%%%%%
\begin{lemma}\label{lemc1}
\refstepcounter{NumConstC}\label{cst1} Let $T>0$. There is $C_{\ref{cst1}}(T)>0$  such that
\begin{align*}
\|L(u(t))\|_1 + \|v(t)\|_2 + \|w(t)\|_{W^{1}_{2}} + \|\partial_t w(t)\|_2 & \le C_{\ref{cst1}}(T)\ , \qquad t\in [0,T]\cap [0,T_m)\ , \\ 
\int_0^t \left[ \|\nabla\sqrt{u}(s)\|_2^2 + \|\nabla\partial_t w(s)\|_2^2 \right]\ \mathrm{d}s & \le C_{\ref{cst1}}(T)\ , \qquad t\in [0,T]\cap [0,T_m)\ ,
\end{align*}
the function $L$ being defined in \eqref{a4}.
\end{lemma}
%%%%%%%%%%%%%%%%

\begin{proof}
By \eqref{a1a} and \eqref{a1d},
\begin{align}
\frac{\mathrm{d}}{\mathrm{d}t} \|L(u)\|_1 & = - \int_\Omega \frac{|\nabla u|^2}{u}\ \mathrm{d}x + \int_\Omega \nabla u \cdot \nabla w\ \mathrm{d}x \nonumber \\
& = - 4 \|\nabla\sqrt{u}\|_2^2 - \int_\Omega u \Delta w\ \mathrm{d}x\ . \label{c2}
\end{align}	
Next, by \eqref{a1b},
\begin{equation*}
\nu \int_\Omega u \partial_t w\ \mathrm{d}x = \nu^2 \varepsilon \int_\Omega \partial_t v \partial_t w \ \mathrm{d}x + \nu \int_\Omega v \partial_t w \ \mathrm{d}x\ .
\end{equation*}
On the one hand, we infer from \eqref{a1c} and \eqref{a1d} that
\begin{equation*}
\int_\Omega v \partial_t w \ \mathrm{d}x = \nu \|\partial_t w\|_2^2 + \frac{1}{2} \frac{\mathrm{d}}{\mathrm{d}t} \left( D \|\nabla w\|_2^2 + \delta \|w\|_2^2 \right)\ .
\end{equation*}
On the other hand, differentiating \eqref{a1c} with respect to time gives 
\begin{equation*}
\int_\Omega \partial_t v \partial_t w \ \mathrm{d}x = D \|\nabla\partial_t w\|_2^2 + \delta \|\partial_t w\|_2^2 + \frac{\nu}{2} \frac{\mathrm{d}}{\mathrm{d}t} \|\partial_t w\|_2^2 \ .
\end{equation*}
Gathering the previous three identities leads us to
\begin{align}
\nu \int_\Omega u \partial_t w\ \mathrm{d}x & = \frac{\nu}{2} \frac{\mathrm{d}}{\mathrm{d}t} \left( D \|\nabla w\|_2^2 + \delta \|w\|_2^2 + \nu^2 \varepsilon \|\partial_t w\|_2^2 \right) \nonumber \\
& \quad + \nu^2 (1+\delta\varepsilon) \|\partial_t w\|_2^2 + \nu^2 \varepsilon D \|\nabla\partial_t w\|_2^2\ . \label{c3}
\end{align}
Now, it follows from \eqref{a1c}, \eqref{c2}, and \eqref{c3} that 
\begin{align}
& \frac{\mathrm{d}}{\mathrm{d}t} \left[ D \|L(u)\|_1 +  \frac{\nu}{2} \left( D \|\nabla w\|_2^2 + \delta \|w\|_2^2 + \nu^2 \varepsilon \|\partial_t w\|_2^2 \right) \right] \nonumber \\
& \qquad\quad + 4D \|\nabla\sqrt{u}\|_2^2 + \nu^2 (1+\delta\varepsilon) \|\partial_t w\|_2^2 + \nu^2 \varepsilon D \|\nabla\partial_t w\|_2^2 \nonumber \\
& \qquad = \int_\Omega u \left( \nu \partial_t w - D \Delta w \right)\ \mathrm{d}x \nonumber \\
& \qquad = \int_\Omega u (v-\delta w)\ \mathrm{d}x\ . \label{c4}
\end{align}
Finally, we deduce from \eqref{a1b} that 
\begin{equation}
\frac{\nu\varepsilon}{2} \frac{\mathrm{d}}{\mathrm{d}t} \|v\|_2^2 + \|v\|_2^2 = \int_\Omega u v\ \mathrm{d}x\ . \label{c5}
\end{equation}
We next recall the Gagliardo-Nirenberg inequality
\begin{equation}
\|z\|_4 \le c_0 \left( \|\nabla z\|_2^{1/2} \|z\|_2^{1/2} + \|z\|_2 \right)\ , \qquad z\in W^{1}_{2}(\Omega)\ , \label{c6}
\end{equation}
where $c_0$ is a positive constant depending only on $\Omega$. Combining \eqref{c1} and \eqref{c6} gives
\begin{align*}
\|u\|_2 & = \|\sqrt{u}\|_4^2 \le 2 c_0^2 \left( \|\nabla\sqrt{u}\|_2 \|\sqrt{u}\|_2 + \|\sqrt{u}\|_2^2 \right) \\
& \le 2 c_0^2 \left( \sqrt{M} \|\nabla\sqrt{u}\|_2 + M \right)\ .
\end{align*}
Consequently, by H\"older's and Young's inequalities, 
\begin{align*}
\int_\Omega u v \ \mathrm{d}x & \le \|u\|_2 \|v\|_2 \le 2 c_0^2 \sqrt{M} \|v\|_2  \|\nabla\sqrt{u}\|_2 + 2 c_0^2 M \|v\|_2 \\
& \le 2D  \|\nabla\sqrt{u}\|_2^2 + C (1+\|v\|_2^2)\ .
\end{align*}
We then infer from \eqref{c4}, \eqref{c5}, the non-negativity of $u$ and $w$, and the previous inequality that
\begin{align*}
& \frac{\mathrm{d}}{\mathrm{d}t} \left[ D \|L(u)\|_1 +  \frac{\nu}{2} \left( \varepsilon \|v\|_2^2 + D \|\nabla w\|_2^2 + \delta \|w\|_2^2 + \nu^2 \varepsilon \|\partial_t w\|_2^2 \right) \right] \\
& \qquad\quad + 2D \|\nabla\sqrt{u}\|_2^2 + \|v\|_2^2 + \nu^2 (1+\delta\varepsilon) \|\partial_t w\|_2^2 + \nu^2 \varepsilon D \|\nabla\partial_t w\|_2^2 \\
& \qquad \le C (1+\|v\|_2^2)\ .
\end{align*}
We finally apply Gronwall's lemma and use the positivity of the parameters to complete the proof.
\end{proof}

An immediate consequence of Lemma~\ref{lemc1} and \eqref{a1c} is the following improved estimate on $w$.

%%%%%%%%%%%%%%%%
\begin{corollary}\label{corc2}
\refstepcounter{NumConstC}\label{cst2} Let $T>0$. There is $C_{\ref{cst2}}(T)>0$  such that
\begin{equation*}
\|w(t)\|_{W^{2}_{2}} \le C_{\ref{cst2}}(T)\ , \qquad t\in [0,T]\cap [0,T_m)\ .
\end{equation*}
\end{corollary}
%%%%%%%%%%%%%%%%

\begin{proof}
By \eqref{a1c}, for $t\in [0,T]\cap [0,T_m)$,
\begin{equation*}
D \|\Delta w(t)\|_2 \le \nu \|\partial_t w(t)\|_2 + \delta \|w(t)\|_2 + \|v(t)\|_2 \le (\nu+\delta+1) C_{\ref{cst1}}(T)\ ,
\end{equation*}
the last estimate being a direct consequence of Lemma~\ref{lemc1}.
\end{proof}

Thanks to Corollary~\ref{corc2}, we are in a position to obtain additional estimates on $u$.

%%%%%%%%%%%%%%%%
\begin{lemma}\label{lemc3}
\refstepcounter{NumConstC}\label{cst3} Let $T>0$ and $r>0$. There are $C_{\ref{cst3}}(T,r)>0$  such that
\begin{equation*}
\|u(t)\|_{r+1} \le C_{\ref{cst3}}(T,r)\ , \qquad t\in [0,T]\cap [0,T_m)\ .
\end{equation*}
\end{lemma}
%%%%%%%%%%%%%%%%

\begin{proof}
Let $T>0$, $r>0$, and $t\in [0,T]\cap [0,T_m)$. By \eqref{a1a} and \eqref{a1d}, 
\begin{align*}
\frac{1}{r+1} \frac{\mathrm{d}}{\mathrm{d}t} \|u\|_{r+1}^{r+1} & = - \int_\Omega u^{r-1} |\nabla u|^2\ \mathrm{d}x + \frac{r}{r+1} \int_\Omega \nabla\left( u^{r+1} \right)\cdot \nabla w\ \mathrm{d}x \\
& = - \frac{4r}{(r+1)^2} \left\| \nabla\left( u^{(r+1)/2} \right) \right\|_2^2 - \frac{r}{r+1} \int_\Omega u^{r+1} \Delta w\ \mathrm{d}x\ .
\end{align*}	
Owing to the Gagliardo-Nirenberg inequality \eqref{c6}, H\"older's inequality, and Corollary~\ref{corc2}, we further obtain
\begin{align*}
\frac{\mathrm{d}}{\mathrm{d}t} \|u\|_{r+1}^{r+1} + \frac{4r}{r+1} \left\| \nabla\left( u^{(r+1)/2} \right) \right\|_2^2 & \le r \|\Delta w\|_2 \left\| u^{(r+1)/2} \right\|_4^2 \\
& \le 2 r c_0^2 C_{\ref{cst2}}(T) \left[ \left\| \nabla\left( u^{(r+1)/2} \right) \right\|_2 \left\| u^{(r+1)/2} \right\|_2 + \left\| u^{(r+1)/2} \right\|_2^2 \right] \\
& \le rC(T) \left[ \left\| \nabla\left( u^{(r+1)/2} \right) \right\|_2 \|u\|_{r+1}^{(r+1)/2} + \|u\|_{r+1}^{r+1} \right]\ .
\end{align*}
We now use Young's inequality to obtain
\begin{equation*}
\frac{\mathrm{d}}{\mathrm{d}t} \|u\|_{r+1}^{r+1} + \frac{2r}{r+1} \left\| \nabla\left( u^{(r+1)/2} \right) \right\|_2^2 \le r(1+r) C(T) \|u\|_{r+1}^{r+1}\ ,
\end{equation*}
from which Lemma~\ref{lemc3} follows after integration with respect to time.
\end{proof}

We are now in a position to apply \cite[Lemma~A.1]{TW2012} and derive an $L^\infty$-estimate on $u$.

%%%%%%%%%%%%%%%%
\begin{lemma}\label{lemc4}
\refstepcounter{NumConstC}\label{cst4} Let $T>0$. There is $C_{\ref{cst4}}(T)>0$  such that
\begin{equation*}
\|u(t)\|_\infty + \|v(t)\|_\infty + \|\partial_t v(t)\|_\infty \le C_{\ref{cst4}}(T)\ , \qquad t\in [0,T]\cap [0,T_m)\ .
\end{equation*}
\end{lemma}
%%%%%%%%%%%%%%%%

\begin{proof}
Let $T>0$. Owing to Corollary~\ref{corc2} and Lemma~\ref{lemc3}, we infer from \cite[Lemma~A.1]{TW2012} that
\begin{equation*}
\|u(t)\|_\infty \le C(T)\ , \qquad t\in [0,T]\cap [0,T_m)\ .
\end{equation*}
Next, \eqref{a1b} and the non-negativity of $u$ and $v$ readily imply that, for $t\in [0,T]\cap [0,T_m)$ and $x\in\Omega$, 
\begin{align*}
0 \le v(t,x) & \le \|v^{in}\|_\infty e^{-t/\nu\varepsilon} + \frac{1}{\nu\varepsilon} \int_0^t \|u(s)\|_\infty e^{-(t-s)/\nu\varepsilon}\ \mathrm{d}s \\
& \le \|v^{in}\|_\infty + \sup_{s\in [0,t]}\{\|u(s)\|_\infty\}\ .
\end{align*}
We finally use \eqref{a1b} to complete the proof.
\end{proof}

We now exploit the properties of the heat semigroup and \eqref{a1c} to derive additional estimates on $w$. 

%%%%%%%%%%%%%%%%
\begin{corollary}\label{corc5}
\refstepcounter{NumConstC}\label{cst5} Let $T>0$. There is $C_{\ref{cst5}}(T)>0$  such that
\begin{equation*}
\|w(t)\|_{W^{1}_{\infty}} \le C_{\ref{cst5}}(T)\ , \qquad t\in [0,T]\cap [0,T_m)\ .
\end{equation*}
\end{corollary}
%%%%%%%%%%%%%%%%

\begin{proof}
Let $T>0$ and $t\in [0,T]\cap [0,T_m)$. Since $W^{11/6}_{3}(\Omega)$ embeds continuously in $W^{1}_{\infty}(\Omega)$ by \cite[Theorem~7.1.2]{Pa1983} and 
\begin{equation*}
\left( L_3(\Omega) , W_{3,\mathcal{B}}^2(\Omega) \right)_{11/12} \stackrel{\cdot}{=} W^{11/6}_{3,\mathcal{B}}(\Omega)
\end{equation*}
by \cite[Theorem~7.2]{Am1993}, it follows from \eqref{a1c}, Duhamel's formula, properties of the heat semigroup, see \cite[Theorem~V.2.1.3]{Am1995}, and Lemma~\ref{lemc4} that
\begin{align*}
\|\nabla w(t)\|_\infty & \le C \| w(t)\|_{W_3^{11/6}} \\
& \le C e^{-\delta t/(2\nu)} \|w^{in}\|_{W_3^{11/6}} + C \int_0^t e^{-\delta(t-s)/(2\nu)} (t-s)^{-11/12} \|v(s)\|_3\ \mathrm{d}s \\
& \le C \|w^{in}\|_{W^{2}_{3}} + C C_{\ref{cst4}}(T) \int_0^t (t-s)^{-11/12}\ \mathrm{d}s \\
& \le C(T)\ .
\end{align*}	
We complete the proof by observing that $\|w(t)\|_\infty \le C(T)$ due to Corollary~\ref{corc2} and the continuous embedding of $W^{2}_{2}(\Omega)$ in $L^\infty(\Omega)$.
\end{proof}

Thanks to the just established $W^{1}_{\infty}$-estimate on $w$, we may derive a $W^{1}_{2}$-estimate on $u$.

%%%%%%%%%%%%%%%%
\begin{lemma}\label{lemc6}
\refstepcounter{NumConstC}\label{cst6} Let $T>0$. There is $C_{\ref{cst6}}(T)>0$  such that
\begin{equation*}
\|\nabla u(t)\|_2 + \int_0^t \|\partial_t u(s)\|_2^2\ \mathrm{d}s \le C_{\ref{cst6}}(T)\ , \qquad t\in [0,T]\cap [0,T_m)\ .
\end{equation*}
\end{lemma}
%%%%%%%%%%%%%%%%

\begin{proof}
Let $T>0$ and $t\in [0,T]\cap [0,T_m)$. It follows from \eqref{a1a} and H\"older's and Young's inequalities that
\begin{align*}
\|\partial_t u\|_2^2 + \frac{1}{2} \frac{\mathrm{d}}{\mathrm{d}t} \|\nabla u\|_2^2 & = - \int_\Omega \partial_t u\ \mathrm{div}(u\nabla w)\ \mathrm{d}x \\
& \le \|\partial_t u\|_2 \|u\|_\infty \|\Delta w\|_2 + \|\partial_t u\|_2 \|\nabla u\|_2 \|\nabla w\|_\infty \\
& \le \frac{1}{2} \|\partial_t u\|_2^2 + \|u\|_\infty^2 \|\Delta w\|_2^2 + \|\nabla u\|_2^2 \|\nabla w\|_\infty^2\ .
\end{align*}
Using Corollary~\ref{corc2}, Lemma~\ref{lemc4}, and Corollary~\ref{corc5}, we further obtain
\begin{equation*}
\|\partial_t u\|_2^2 + \frac{\mathrm{d}}{\mathrm{d}t} \|\nabla u\|_2^2 \le 2 C_{\ref{cst4}}(T)^2 C_{\ref{cst2}}(T)^2 + 2 C_{\ref{cst5}}(T)^2 \|\nabla u\|_2^2\ ,
\end{equation*}
from which Lemma~\ref{lemc6} readily follows.
\end{proof}

As for $w$ in Corollary~\ref{corc5}, the properties of the heat semigroup and \eqref{a1a} provide additional estimates on $u$. 

%%%%%%%%%%%%%%%%
\begin{corollary}\label{corc7}
\refstepcounter{NumConstC}\label{cst7} Let $T>0$. There is $C_{\ref{cst7}}(T)>0$ such that
\begin{equation*}
\|u(t)\|_{W^{1}_{3}} + \|v(t)\|_{W^{1}_{3}} \le C_{\ref{cst7}}(T)\ , \qquad t\in [0,T]\cap [0,T_m)\ .
\end{equation*}
\end{corollary}
%%%%%%%%%%%%%%%%

\begin{proof}
Let $T>0$ and $t\in [0,T]\cap [0,T_m)$. We infer from \eqref{a1a}, Duhamel's formula, properties of the heat semigroup, see \cite[Proposition~12.5]{Am1983} and \cite[Theorem~V.2.1.3]{Am1995}, and H\"older's inequality that
\begin{align*}
\|\nabla u(t)\|_3 & \le C \|\nabla u^{in}\|_3 + C \int_0^t (t-s)^{-2/3} \left\| \mathrm{div} (u\nabla w)(s) \right\|_2\ \mathrm{d}s \\
& \le C + C \int_0^t (t-s)^{-2/3} \|\nabla w(s)\|_\infty \|\nabla u(s)\|_2\ \mathrm{d}s \\
& \qquad + C \int_0^t (t-s)^{-2/3} \|u(s)\|_\infty \|\Delta w(s)\|_2\ \mathrm{d}s\ .
\end{align*}
Owing to Corollary~\ref{corc2}, Lemma~\ref{lemc4}, Corollary~\ref{corc5}, and Lemma~\ref{lemc6}, we further obtain
\begin{align}
\|\nabla u(t)\|_3 & \le C + C C_{\ref{cst5}}(T) C_{\ref{cst6}}(T) \int_0^t (t-s)^{-2/3}\ \mathrm{d}s \nonumber \\
& \qquad + C C_{\ref{cst2}}(T) C_{\ref{cst4}}(T) \int_0^t (t-s)^{-2/3}\ \mathrm{d}s \nonumber \\
& \le C(T)\ . \label{c7}
\end{align}
As 
\begin{equation*}
\nabla v(t,x) = \nabla v^{in}(x) e^{-t/\nu\varepsilon} + \frac{1}{\nu\varepsilon} \int_0^t e^{-(t-s)/\nu\varepsilon} \nabla u(s,x)\ \mathrm{d}s\ , \qquad x\in\Omega\ ,
\end{equation*}
by \eqref{a1b}, we readily deduce from \eqref{c7} that
\begin{equation*}
\|\nabla v(t)\|_3 \le C(T)\ .
\end{equation*}
Combining the previous estimate and \eqref{c7} with Lemma~\ref{lemc3} for $r=2$ and Lemma~\ref{lemc4} ends the proof.
\end{proof}

\refstepcounter{NumConstC}\label{cst8} Summarizing the outcome of the above analysis, we have shown that, for any $T>0$, there is  $C_{\ref{cst8}}(T)>0$ such that
\begin{equation}
\|u(t)\|_{W^{1}_{3}} + \|v(t)\|_{W^{1}_{3}} + \|w(t)\|_{W^{1}_{3}} \le C_{\ref{cst8}}(T)\ , \qquad t\in [0,T]\cap [0,T_m)\ , \label{c8}
\end{equation}
see Corollary~\ref{corc5} and Corollary~\ref{corc7}. According to \eqref{b2}, excluding finite time blowup requires uniformly continuous 
estimates with respect to time which we derive now. In fact, we shall establish H\"older estimates with respect to time. 

%%%%%%%%%%%%%%%%
\begin{lemma}\label{lemc8}
\refstepcounter{NumConstC}\label{cst9} Let $T>0$. There is $C_{\ref{cst9}}(T)>0$ such that, for $t_1\in [0,T]\cap [0,T_m)$ and $t_2\in [0,T]\cap [0,T_m)$, 
\begin{equation*}
\|u(t_2)-u(t_1)\|_{W^{1}_{3}} + \|v(t_2)-v(t_1)\|_{W^{1}_{3}} + \|w(t_2)-w(t_1)\|_{W^{1}_{3}} \le C_{\ref{cst9}}(T) |t_2-t_1|^\alpha\ ,
\end{equation*}
where $\alpha := (3\theta-2)/6\theta$.
\end{lemma}
%%%%%%%%%%%%%%%%

\begin{proof}
Let $T>0$, $t_1\in [0,T]\cap [0,T_m)$, and $t_2\in [0,T]\cap [0,T_m)$, $t_2>t_1$. It first follows from \eqref{a1b} and \eqref{c8} that 
\begin{equation*}
\|\partial_t v(t)\|_{W^{1}_{3}} \le \frac{1}{\nu\varepsilon} \left( \|u(t)\|_{W^{1}_{3}} + \|v(t)\|_{W^{1}_{3}} \right) \le C(T)\ , \qquad t\in [0,T]\cap [0,T_m)\ .
\end{equation*}
Consequently, 
\begin{equation}
 \|v(t_2)-v(t_1)\|_{W^{1}_{3}} \le \int_{t_1}^{t_2} \|\partial_t v(s)\|_{W^{1}_{3}}\ \mathrm{d}s \le C(T) (t_2-t_1)\ . \label{c9}
\end{equation}
Next, by H\"older's inequality,
\begin{align*}
\|w(t_2)-w(t_1)\|_{W^{1}_{3}} & \le C \|w(t_2)-w(t_1)\|_{W^{1}_{\infty}}^{1/3} \|w(t_2)-w(t_1)\|_{W^{1}_{2}}^{2/3} \\
& \le C \left( \|w(t_2)\|_{W^{1}_{\infty}} + \|w(t_1)\|_{W^{1}_{\infty}} \right)^{1/3} \left( \int_{t_1}^{t_2} \|\partial_t w(s)\|_{W^{1}_{2}}\ \mathrm{d}s \right)^{2/3}\ .
\end{align*}
We then deduce from Lemma~\ref{lemc1}, Corollary~\ref{corc5}, and H\"older's inequality that
\begin{align}
\|w(t_2)-w(t_1)\|_{W^{1}_{3}} & \le C (2C_{\ref{cst5}}(T))^{1/3} (t_2-t_1)^{1/3} \left( \int_{t_1}^{t_2} \|\partial_t w(s)\|_{W^{1}_{2}}^2\ \mathrm{d}s \right)^{1/3} \nonumber \\
& \le C(T) \left[ C_{\ref{cst1}}(T)^2 (t_2-t_1) + C_{\ref{cst1}}(T) \right]^{1/3} (t_2-t_1)^{1/3} \nonumber \\
& \le C(T) (t_2-t_1)^{1/3}\ . \label{c10}
\end{align}
Finally, as in the proof of Corollary~\ref{corc7}, we infer from \eqref{a1a}, Duhamel's formula, properties of the heat semigroup \cite[Theorem~V.2.1.3]{Am1995}, and H\"older's inequality that, for $t\in [0,T]\cap [0,T_m)$, 
\begin{align*}
\|u(t)\|_{W^{2\theta}_{2}} & \le C \|u^{in}\|_{W^{2\theta}_{2}} + C \int_0^t (t-s)^{-\theta} \|\mathrm{div}(u\nabla w)(s)\|_2\ \mathrm{d}s \\ 
& \le C + C \int_0^t (t-s)^{-\theta} \|\nabla u(s)\|_2 \|\nabla w(s)\|_\infty\ \mathrm{d}s \\ 
& \qquad + C \int_0^t (t-s)^{-\theta} \|u(s)\|_\infty \|\Delta w(s)\|_2\ \mathrm{d}s \ . 
\end{align*}
Recalling that $\theta\in (5/6,1)$, it follows from Corollary~\ref{corc2}, Lemma~\ref{lemc4}, Corollary~\ref{corc5}, and Lemma~\ref{lemc6} that
\begin{align}
\|u(t)\|_{W^{2\theta}_{2}} & \le C + C C_{\ref{cst6}}(T) C_{\ref{cst5}}(T) \int_0^t (t-s)^{-\theta}\ \mathrm{d}s \nonumber \\
& \qquad + C C_{\ref{cst4}}(T) C_{\ref{cst2}}(T) \int_0^t (t-s)^{-\theta}\ \mathrm{d}s \nonumber \\
& \le C(T)\ . \label{c11}
\end{align}
Now, since $W^{4/3}_{2}(\Omega)$ is continuously embedded in $W^{1}_{3}(\Omega)$, interpolation inequalities and \eqref{c11} entail that
\begin{align*}
\|u(t_2) - u(t_1)\|_{W^{1}_{3}} & \le C \|u(t_2) - u(t_1)\|_{W^{4/3}_{2}} \\
& \le C \|u(t_2) - u(t_1)\|_{W^{2\theta}_{2}}^{2/3\theta} \|u(t_2) - u(t_1)\|_2^{(3\theta-2)/3\theta} \\
& \le \left( \|u(t_2)\|_{W^{2\theta}_{2}} + \|u(t_1)\|_{W^{2\theta}_{2}} \right)^{2/3\theta} \left( \int_{t_1}^{t_2} \|\partial_t u(s)\|_2\ \mathrm{d}s \right)^{(3\theta-2)/3\theta} \\
& \le C(T) (t_2-t_1)^{\alpha} \left( \int_{t_1}^{t_2} \|\partial_t u(s)\|_2^2\ \mathrm{d}s \right)^{\alpha} \ .
\end{align*}
Combining the previous inequality with Lemma~\ref{lemc6} and recalling \eqref{c9} and \eqref{c10} complete the proof, after noticing that $0<\alpha<1/3<1$.
\end{proof}

\begin{proof}[Proof of Theorem~\ref{thma1}]
The well-posedness in $W^{1}_{3}$ on some maximal time interval $[0,T_m)$ is provided by Proposition~\ref{propb1}. According to Lemma~\ref{lemc8}, the condition \eqref{b2} of Proposition~\ref{propb1} is satisfied for any $T>0$, and we thus conclude that $T_m=\infty$.
\end{proof}

%%%%%%%%%%%%%%%%
%%%%%%%%%%%%%%%%
\section{Bounded solutions}\label{s3}
%%%%%%%%%%%%%%%%
%%%%%%%%%%%%%%%%

\newcounter{NumConstB}

Let $\theta\in (5/6,1)$. Consider initial conditions $(u^{in},v^{in},w^{in})\in \mathcal{I}_{M,\theta}$ and denote the corresponding solution to \eqref{a1}-\eqref{a2} given in Theorem~\ref{thma1} by $(u,v,w)$. We begin with the evolution of the $L_1$-norms of $u$, $v$, and $w$.

%%%%%%%%%%%%%%%%
\begin{lemma}\label{lemd1}
For $t\ge 0$,
\begin{align}
\|u(t)\|_1 & = M = \|u^{in}\|_1\ , \label{d1} \\
\|v(t)\|_1 & \le \|v^{in}\|_1 + M\ , \label{d2} \\
\|w(t)\|_1 & \le \|w^{in}\|_1 + \frac{M}{\delta} + b_0 (\|v^{in}\|_1 + M)\ , \label{d3}
\end{align}
where $b_0 := \varepsilon/|\delta\varepsilon-1|$ if $\delta\varepsilon\ne 1$ and $b_0 := 1/\delta e$ when $\delta\varepsilon=1$.
\end{lemma}
%%%%%%%%%%%%%%%%

\begin{proof} 
The identity \eqref{d1} is nothing but \eqref{b1}, while it readily follows from \eqref{a1b} and \eqref{d1} that
\begin{equation}
\|v(t)\|_1  = \|v^{in}\|_1 e^{-t/\nu\varepsilon} + M \left( 1 - e^{-t/\nu\varepsilon} \right)\ , \qquad t\ge 0\ . \label{d2b}
\end{equation}
The upper bound \eqref{d2} is then an immediate consequence of \eqref{d2b}. 
Finally, by \eqref{a1c} and \eqref{a1d}, 
\begin{equation*}
\nu \frac{\mathrm{d}}{\mathrm{d}t} \|w\|_1 + \delta \|w\|_1 = \|v\|_1\ , \qquad t\ge 0\ ,
\end{equation*}
which allows us to compute explicitly the time evolution of the $L_1$-norm of $w$ and find, for $t\ge 0$, 
\begin{subequations} \label{d4}
\begin{equation}
\|w(t)\|_1 = \frac{M}{\delta} + \left( \|w^{in}\|_1 - \frac{M}{\delta} \right) e^{-\delta t/\nu} + \frac{\varepsilon}{\delta\varepsilon-1} \left( \|v^{in}\|_1 -M \right) \left( e^{-t/\nu\varepsilon} - e^{-\delta t/\nu} \right) \label{d4a}
\end{equation} 
when $\delta\varepsilon\ne 1$ and 
\begin{equation}
\|w(t)\|_1 = \frac{M}{\delta} + \left( \|w^{in}\|_1 - \frac{M}{\delta} \right) e^{-\delta t/\nu} + \frac{\|v^{in}\|_1 -M}{\nu} t e^{-\delta t/\nu} \label{d4b}
\end{equation} 
when $\delta\varepsilon=1$. 
\end{subequations}
Since $ze^{-z}\le 1/e$ for $z\ge 0$, the upper bound \eqref{d3} readily follows from \eqref{d4}.
\end{proof}

%%%%%%%%%%%%%%%%
%%%%%%%%%%%%%%%%
\subsection{A Liapunov functional}\label{s3.1}
%%%%%%%%%%%%%%%%
%%%%%%%%%%%%%%%%

We next turn to the availability of a Liapunov function and recall that $\mathcal{E}_0$ and $\mathcal{E}$ are defined by
\begin{align*}
\mathcal{E}_0(u,w) & := \int_\Omega \left[ L(u) - uw \right]\ \mathrm{d}x + \frac{D}{2} \|\nabla w\|_2^2 + \frac{\delta}{2} \|w\|_2^2\ , \\ 
\mathcal{E}(u,v,w) & := \mathcal{E}_0(u,w) + \frac{\varepsilon}{2} \| -D \Delta w + \delta w - v \|_2^2\ , 
\end{align*}
see \eqref{a3}, the function $L$ being given by \eqref{a4}.

%%%%%%%%%%%%%%%%
\begin{lemma}\label{lemd2}
For $t\ge 0$, there holds
\begin{equation*}
\frac{\mathrm{d}}{\mathrm{d}t} \mathcal{E}(u(t),v(t),w(t)) + \mathcal{D}(u(t),v(t),w(t)) = 0\ , \qquad t\ge 0\ ,
\end{equation*}
where 
\begin{align}
\mathcal{D}(u,v,w) & := \int_\Omega u |\nabla(\ln{u}-w)|^2\ \mathrm{d}x + \frac{1+\delta\varepsilon}{\nu} \|-D\Delta w + \delta w - v\|_2^2 \nonumber \\
& \qquad + \frac{\varepsilon D}{\nu} \|\nabla(-D\Delta w + \delta w - v)\|_2^2\ . \label{d5}
\end{align}
\end{lemma}
%%%%%%%%%%%%%%%%

\begin{proof}
By \eqref{a1}, 
\begin{align*}
\frac{\mathrm{d}}{\mathrm{d}t} \mathcal{E}_0(u,w) & = \int_\Omega (\ln{u} - w) \partial_t u\ \mathrm{d}x - \int_\Omega u \partial_t w\ \mathrm{d}x + \int_\Omega \left( -D \Delta w + \delta w \right) \partial_t w \ \mathrm{d}x \\
& = - \int_\Omega u |\nabla(\ln{u}-w)|^2\ \mathrm{d}x - \int_\Omega (\varepsilon\nu \partial_t v + v) \partial_t w\ \mathrm{d}x + \int_\Omega \left( v - \nu \partial_t w \right) \partial_t w \ \mathrm{d}x \\
& = - \int_\Omega u |\nabla(\ln{u}-w)|^2\ \mathrm{d}x - \nu \|\partial_t w\|_2^2 - \varepsilon\nu \int_\Omega \partial_t v \partial_t w \ \mathrm{d}x\ .
\end{align*}
Owing to \eqref{a1c}, 
\begin{align*}
\int_\Omega \partial_t v \partial_t w \ \mathrm{d}x & = \int_\Omega \left( \nu \partial^2_t w - D \Delta\partial_t w + \delta \partial_t w \right) \partial_t w \ \mathrm{d}x \\
& = \frac{\nu}{2} \frac{\mathrm{d}}{\mathrm{d}t} \|\partial_t w\|_2^2 + D \|\nabla\partial_t w\|_2^2 + \delta \|\partial_t w\|_2^2\ .
\end{align*}
Combining the previous two identities and using \eqref{a1c} to replace $\partial_t w$ complete the proof.
\end{proof}

Throughout the remainder of Section~\ref{s3}, $b$ and $(b_i)_{i\ge 1}$ are positive constants depending only on $\Omega$, $\nu$, $\varepsilon$, $D$, $\delta$, $\theta$, $u^{in}$, $v^{in}$, and $w^{in}$. Dependence upon additional parameters will be indicated explicitly.

%%%%%%%%%%%%%%%%
%%%%%%%%%%%%%%%%
\subsection{Time-independent estimates}\label{s3.2}
%%%%%%%%%%%%%%%%
%%%%%%%%%%%%%%%%

As in \cite{BN1993, GZ1998, NSY1997}, we exploit the structure of the Liapunov functional $\mathcal{E}$ and first show that it is bounded from below as soon as $M$ is suitably small. To this end, we first recall the following consequence of the Moser-Trudinger inequality \cite[Proposition~2.3]{CY1988}, see \cite[Corollary~2.6]{GZ1998}. 

%%%%%%%%%%%%%%%%
\begin{proposition}\label{propd2}
There is $K_0>0$ depending only on $\Omega$ such that, for $z\in W_2^1(\Omega)$, 
\begin{equation*}
\int_\Omega e^{|z|}\ \mathrm{d}x \le K_0 \exp{\left( \frac{\|\nabla z\|_2^2}{8\pi} + \frac{\|z\|_1}{|\Omega|} \right)}\ .
\end{equation*}
\end{proposition}
%%%%%%%%%%%%%%%%

We now derive upper and lower bounds on $\mathcal{E}(u,v,w)$.

%%%%%%%%%%%%%%%%
\begin{lemma}\label{lemd3}
\refstepcounter{NumConstB}\label{cstb1} There is $b_{\ref{cstb1}}>0$ such that, for $t\ge 0$, 
\begin{align}
\mathcal{E}(u(t),v(t),w(t)) & \le \mathcal{E}(u^{in},v^{in},w^{in})\ , \label{d6} \\
\mathcal{E}(u(t),v(t),w(t)) & \ge \frac{4\pi D - M}{8\pi} \|\nabla w(t)\|_2^2 + \frac{\delta}{2} \|w(t)\|_2^2 \nonumber \\
& \qquad + \frac{\varepsilon}{2} \|-D\Delta w(t) + \delta w(t) - v(t)\|_2^2 - b_{\ref{cstb1}}\ . \label{d7}
\end{align}
\end{lemma}
%%%%%%%%%%%%%%%%

\begin{proof}
We first argue as in the proof of \cite[Lemma~4.5]{GZ1998}, see also \cite[Theorem~2~(iv)]{BN1993} and \cite[Lemma~3.4]{NSY1997}, to obtain that 
\begin{equation}
\mathcal{E}_0(u,w) \ge \frac{D}{2} \|\nabla w\|_2^2 + \frac{\delta}{2} \|w\|_2^2 - M \ln({\| e^w\|_1}) + M \ln{M} - M + |\Omega|\ . \label{d7a}
\end{equation}
It then follows from Proposition~\ref{propd2} that
\begin{equation*}
\mathcal{E}_0(u,w) \ge \frac{4\pi D - M}{8\pi} \|\nabla w\|_2^2 + \frac{\delta}{2} \|w\|_2^2 - M \ln{K_0} - \frac{M}{|\Omega|} \|w\|_1 - 1\ .
\end{equation*}
Combining the previous inequality with \eqref{d3} gives
\begin{equation*}
\mathcal{E}_0(u,w) \ge \frac{4\pi D - M}{8\pi} \|\nabla w\|_2^2 + \frac{\delta}{2} \|w\|_2^2 - b_{\ref{cstb1}}\ .
\end{equation*}
Consequently, 
\begin{equation*}
\mathcal{E}(u,v,w) \ge \frac{4\pi D - M}{8\pi} \|\nabla w\|_2^2 + \frac{\delta}{2} \|w\|_2^2 + \frac{\varepsilon}{2} \|-D\Delta w + \delta w - v\|_2^2 - b_{\ref{cstb1}}\ ,
\end{equation*}
and we have shown the expected lower bound. The upper bound is a straightforward consequence of Lemma~\ref{lemd2} and the non-negativity of $\mathcal{D}$.
\end{proof}

We are now in a position to deduce a first set of time-independent estimates, provided $M\in (0,4\pi D)$.

%%%%%%%%%%%%%%%%
\begin{lemma}\label{lemd4}
\refstepcounter{NumConstB}\label{cstb2} Assume that $M\in (0,4\pi D)$. There is $b_{\ref{cstb2}}>0$ such that, for $t\ge 0$,
\begin{equation*}
\|u(t)\ln{u(t)}\|_1 + \|w(t)\|_{W_2^1} + \|\partial_t w(t)\|_2 + \int_0^\infty \|\partial_t w(s)\|_{W_2^1}^2\ \mathrm{d}s \le b_{\ref{cstb2}}\ .
\end{equation*}
\end{lemma}
%%%%%%%%%%%%%%%%

\begin{proof}
Let $t\ge 0$. It readily follows from \eqref{a1c}, \eqref{d6}, and \eqref{d7} that
\begin{equation*}
\min{\left\{ \frac{4\pi D - M}{8\pi} , \frac{\delta}{2} , \frac{\nu^2 \varepsilon}{2} \right\}} \left[ \|w(t)\|_{W_2^1}^2 + \|\partial_t w(t)\|_2^2 \right] \le b_{\ref{cstb1}} + \mathcal{E}(u^{in},v^{in},w^{in})\ .
\end{equation*}
Hence, since $M\in (0,4\pi D)$,
\begin{equation}
 \|w(t)\|_{W_2^1} + \|\partial_t w(t)\|_2 \le b\ , \qquad t\ge 0\ . \label{d8}
\end{equation}
Next, thanks to the non-negativity and convexity of the function $L$ defined in \eqref{a4}, it follows from  \eqref{a3}, \eqref{d6}, and Young's inequality that
\begin{align*}
\| L(u(t))\|_1 & = \int_\Omega L(u(t,x))\ \mathrm{d}x \le \mathcal{E}_0(u(t),w(t)) + \int_\Omega u(t,x) w(t,x)\ \mathrm{d}x \\
& \le \mathcal{E}(u(t),v(t),w(t)) + \frac{1}{2} \int_\Omega \left[ L(u(t,x)) + L^*(2w(t,x)) \right]\ \mathrm{d}x \\
& \le \mathcal{E}(u^{in},v^{in},w^{in}) + \frac{1}{2} \|L(u(t))\|_1 + \frac{1}{2} \int_\Omega L^*(2w(t,x))\ \mathrm{d}x\ ,
\end{align*}
where $L^*$ is the convex conjugate of $L$; that is, $L^*(z) := e^z -1$, $z\in\mathbb{R}$. Since
\begin{equation*}
\int_\Omega L^*(2w(t,x))\ \mathrm{d}x \le \int_\Omega e^{2w(t,x)}\ \mathrm{d}x \le K_0 \exp{\left( \frac{\|\nabla w(t)\|_2^2}{2\pi} + \frac{2}{|\Omega|} \|w(t)\|_1 \right)} \le b
\end{equation*}
by \eqref{d3}, \eqref{d8}, and Proposition~\ref{propd2}, we end up with 
\begin{equation}
\|L(u(t))\|_1 \le 2 \mathcal{E}(u^{in},v^{in},w^{in}) + b\ , \qquad t\ge 0\ . \label{d9}
\end{equation}
An immediate consequence of \eqref{d9} and the elementary inequality $z|\ln{z}|\le L(z) + z + 1$ for $z\ge 0$ is an $L\ln{L}$-estimate for $u$:
\begin{equation}
\|u(t)\ln{u(t)}\|_1 \le b\ , \qquad t\ge 0\ . \label{d10}
\end{equation}
We finally infer from \eqref{a1c}, \eqref{d6}, \eqref{d7}, and Lemma~\ref{lemd2} that
\begin{align*}
& \nu \min{\left\{ (1+\delta\varepsilon) , \varepsilon D  \right\}} \int_0^t \|\partial_t w(s)\|_{W_2^1}^2\ \mathrm{d}s \\
& \qquad \le \int_0^t \left( \nu(1+\delta\varepsilon) \|\partial_t w(s)\|_2^2 + \nu\varepsilon D \|\nabla \partial_t w(s)\|_2^2 \right)\ \mathrm{d}s \\ 
& \qquad \le \int_0^t \mathcal{D}(u(s),v(s),w(s))\ \mathrm{d}s \\
& \qquad \le \mathcal{E}(u^{in},v^{in},w^{in}) - \mathcal{E}(u(t),v(t),w(t)) \\
& \qquad \le \mathcal{E}(u^{in},v^{in},w^{in}) + b_{\ref{cstb1}}\ ,
\end{align*}
which, together with \eqref{d8} and \eqref{d10}, complete the proof of Lemma~\ref{lemd4}.
\end{proof}

We next derive $L_r$-estimates for $(u,v)$, the starting point being the just proved $L\ln{L}$-estimate on $u$ and the following fundamental inequality established in \cite[Equation~(22)]{BHN1994}: given $\eta>0$, there is a positive constant $\kappa_\eta>0$ depending only on $\eta$ and $\Omega$ such that
\begin{equation}
\|z\|_3^3 \le \eta \|z\|_{W_2^1}^2 \|z\ln{|z|}\|_1 + \kappa_\eta \|z\|_1\ , \qquad z\in W_2^1(\Omega)\ . \label{d11}
\end{equation}

%%%%%%%%%%%%%%%%
\begin{lemma}\label{lemd5}
\refstepcounter{NumConstB}\label{cstb3} Assume that $M\in (0,4\pi D)$. There is $b_{\ref{cstb3}}>0$ such that, for $t\ge 0$,
\begin{equation*}
\|u(t)\|_2 + \|v(t)\|_3 \le b_{\ref{cstb3}}\ .
\end{equation*}
\end{lemma}
%%%%%%%%%%%%%%%%

\begin{proof}
On the one hand, it follows from \eqref{a1a}, \eqref{a1d}, H\"older's inequality, and the non-negativity of $u$ and $w$ that 
\begin{align*}
\frac{1}{2} \frac{\mathrm{d}}{\mathrm{d}t} \|u\|_2^2 & =- \|\nabla u\|_2^2 - \frac{1}{2} \int_\Omega u^2 \Delta w\ \mathrm{d}x \\
& = - \|\nabla u\|_2^2 + \frac{1}{2D} \int_\Omega u^2 \left( v - \delta w - \nu \partial_t w \right)\ \mathrm{d}x \\
& \le - \|\nabla u\|_2^2 + \frac{1}{2D} \|v\|_3 \|u\|_3^2 + \frac{\nu}{2D} \|\partial_t w\|_2 \|u\|_4^2\ .
\end{align*}
Using the Gagliardo-Nirenberg inequality \eqref{c6} and Young's inequality, we further obtain
\begin{align*}
\frac{\mathrm{d}}{\mathrm{d}t} \|u\|_2^2 + 2\|\nabla u\|_2^2 & \le \frac{1}{D} \|v\|_3 \|u\|_3^2 + \frac{\nu }{D} \|\partial_t w\|_2 \|u\|_4^2 \nonumber \\
& \le \frac{1}{3} \|v\|_3^3 + \frac{2}{3 D^{3/2}} \|u\|_3^3 + \frac{2 \nu c_0^2}{D} \|\partial_t w\|_2 \left( \|\nabla u\|_2 \|u\|_2 + \|u\|_2^2 \right) \nonumber \\
& \le \frac{1}{3} \|v\|_3^3 + \frac{2}{3 D^{3/2}} \|u\|_3^3 + \frac{1}{2} \|\nabla u\|_2^2 + b \left( \|\partial_t w\|_2 + \|\partial_t w\|_2^2 \right) \|u\|_2^2\ . 
\end{align*}
Hence, thanks to Lemma~\ref{lemd4},
\begin{equation}
\frac{\mathrm{d}}{\mathrm{d}t} \|u\|_2^2 + \frac{3}{2} \|\nabla u\|_2^2 \le  \frac{1}{3} \|v\|_3^3 + \frac{2}{3 D^{3/2}} \|u\|_3^3 + b \|\partial_t w\|_2 \|u\|_2^2\ . \label{d12}
\end{equation}
On the other hand, we infer from \eqref{a1b} and H\"older's and Young's inequalities that
\begin{equation*}
\frac{\nu\varepsilon}{3} \frac{\mathrm{d}}{\mathrm{d}t} \|v\|_3^3 + \|v\|_3^3 \le \|u\|_3 \|v\|_3^2 \le \frac{1}{3} \|u\|_3^3 + \frac{2}{3} \|v\|_3^3\ .
\end{equation*}
Hence,
\begin{equation}
\nu\varepsilon \frac{\mathrm{d}}{\mathrm{d}t} \|v\|_3^3 + \|v\|_3^3 \le  \|u\|_3^3 \ . \label{d13}
\end{equation}
Introducing $Y:= \|u\|_2^2 + \nu\varepsilon \|v\|_3^3$, we sum up \eqref{d12} and \eqref{d13} and use \eqref{d1},  \eqref{d11}, and Lemma~\ref{lemd4} to obtain, for $\eta>0$,
\begin{align*}
\frac{\mathrm{d}Y}{\mathrm{d}t} + \frac{3}{2} \|\nabla u\|_2^2 + \frac{2}{3} \|v\|_3^3 & \le \left( 1 + \frac{2}{3 D^{3/2}} \right) \|u\|_3^3 + b \|\partial_t w\|_2 \|u\|_2^2 \\
& \le \left( 1 + \frac{2}{3 D^{3/2}} \right) \left( \eta b_{\ref{cstb2}} \|u\|_{W_2^1}^2 + \kappa_\eta M \right) + b \|\partial_t w\|_2 \|u\|_2^2 \\
& \le b \left( \eta \|\nabla u\|_{2}^2 + \eta \|u\|_2^2 + \|\partial_t w\|_2 \|u\|_2^2 \right)\ .
\end{align*}
By Poincar\'e's inequality, there is $K_1>0$ depending only on $\Omega$ such that
\begin{equation*}
\left\| z - \frac{1}{|\Omega|} \int_\Omega z(x)\ \mathrm{d}x \right\|_2 \le K_1 \|\nabla z\|_2\ , \qquad z\in W_2^1(\Omega)\ ,
\end{equation*}
which implies that
\begin{equation}
\|z\|_2^2 \le 2 K_1^2 \|\nabla z\|_2^2 + \frac{2}{|\Omega|} \|z\|_1^2\ , \qquad z\in W_2^1(\Omega)\ . \label{d14}
\end{equation}
Consequently, it follows from \eqref{d1} and \eqref{d14} that
\begin{equation*}
\frac{\mathrm{d}Y}{\mathrm{d}t} + \frac{3}{2} \|\nabla u\|_2^2 + \frac{2}{3} \|v\|_3^3 \le b \left( \eta \|\nabla u\|_{2}^2 + 2\eta K_1^2 \|\nabla u\|_2^2 + 2\eta \frac{M^2}{|\Omega|} + \|\partial_t w\|_2 \|u\|_2^2 \right)\ .
\end{equation*}
Choosing $\eta = 1/[2b(1+2K_1^{2})]$, we end up with 
\begin{equation*}
\frac{\mathrm{d}Y}{\mathrm{d}t} + \|\nabla u\|_2^2 + \frac{2}{3} \|v\|_3^3 \le b \left( 1 + \|\partial_t w\|_2 \|u\|_2^2 \right)\ .
\end{equation*}
Using \eqref{d1} and \eqref{d14} again, we deduce that
\begin{equation*}
\frac{\mathrm{d}Y}{\mathrm{d}t} + \frac{1}{2K_1^2} \left( \|u\|_2^2 - \frac{2M^2}{|\Omega|} \right) + \frac{2}{3} \|v\|_3^3 \le b \left( 1 + \|\partial_t w\|_2 \|u\|_2^2 \right)\ .
\end{equation*}
\refstepcounter{NumConstB}\label{cstb4} \refstepcounter{NumConstB}\label{cstb5} Therefore, there are $b_{\ref{cstb4}}>0$ and $b_{\ref{cstb5}}>0$ such that
\begin{equation*}
\frac{\mathrm{d}Y}{\mathrm{d}t} + b_{\ref{cstb4}} Y \le b_{\ref{cstb5}} \left( 1 + \|\partial_t w\|_2 \|u\|_2^2 \right) \le b_{\ref{cstb5}} \left( 1 + \|\partial_t w\|_2 Y \right)\ . 
\end{equation*}
Equivalently,
\begin{equation*}
\frac{\mathrm{d}}{\mathrm{d}t} \left[ Y(t) \exp{\left( b_{\ref{cstb4}} t - b_{\ref{cstb5}} \int_0^t \|\partial_t w(s)\|_2\ \mathrm{d}s \right)} \right] \le b_{\ref{cstb5}} \exp{\left( b_{\ref{cstb4}} t - b_{\ref{cstb5}} \int_0^t \|\partial_t w(s)\|_2\ \mathrm{d}s \right)}\ ,
\end{equation*}
and, after integration with respect to time,
\begin{align}
Y(t) & \le Y(0) \exp{\left( b_{\ref{cstb5}} \int_0^t \|\partial_t w(s)\|_2\ \mathrm{d}s - b_{\ref{cstb4}} t \right)} \nonumber\\
& \qquad + b_{\ref{cstb5}} \int_0^t \exp{\left( b_{\ref{cstb5}} \int_s^t \|\partial_t w(s_*)\|_2\ \mathrm{d}s_* - b_{\ref{cstb4}} (t-s) \right)}\ \mathrm{d}s \ . \label{d15}
\end{align}
Owing to Lemma~\ref{lemd4}, we infer from Young's inequality that, for $0\le s\le t$, 
\begin{align}
b_{\ref{cstb5}} \int_s^t \|\partial_t w(s_*)\|_2\ \mathrm{d}s_* - b_{\ref{cstb4}} (t-s) & \le \int_s^t \left( \frac{b_{\ref{cstb4}}}{2} + \frac{b_{\ref{cstb5}}^2}{2 b_{\ref{cstb4}}} \|\partial_t w(s_*)\|_2^2 \right)\ \mathrm{d}s_* - b_{\ref{cstb4}} (t-s) \nonumber \\
& \le \frac{b_{\ref{cstb2}} b_{\ref{cstb5}}^2}{2 b_{\ref{cstb4}}} - \frac{b_{\ref{cstb4}}}{2} (t-s) \ . \label{d16}
\end{align}
Combining \eqref{d15} and \eqref{d16} gives
\begin{equation*}
Y(t) \le b e^{-b_{\ref{cstb4}}t/2} + b \int_0^t e^{-b_{\ref{cstb4}}(t-s)/2}\ \mathrm{d}s \le b\ ,
\end{equation*}
and completes the proof.
\end{proof}

Owing to Lemma~\ref{lemd5}, we may argue as in the proof of Corollary~\ref{corc5}, using in addition the exponential decay due to the positivity of $\delta$, to obtain a Lipschitz estimate on $w$.

%%%%%%%%%%%%%%%%
\begin{corollary}\label{cord6}
\refstepcounter{NumConstB}\label{cstb6} Assume that $M\in (0,4\pi D)$. There is $b_{\ref{cstb6}}>0$ such that, for $t\ge 0$,
\begin{equation*}
\|w(t)\|_{W_\infty^1} \le b_{\ref{cstb6}}\ .
\end{equation*}
\end{corollary}
%%%%%%%%%%%%%%%%

\begin{proof}[Proof of Theorem~\ref{thma2}~(a)]
From Lemma~\ref{lemd5} and Corollary~\ref{cord6}, we deduce that 
\begin{equation*}
u\nabla w \in L_\infty((0,\infty);L_2(\Omega;\mathbb{R}^2))\ .
\end{equation*} 
Therefore, since $u^{in}\in L_\infty(\Omega)$, we are in a position to apply \cite[Lemma~A.1]{TW2012} to conclude that
\begin{equation*}
\|u(t)\|_\infty \le b\ , \qquad t\ge 0\ ,
\end{equation*}
see also \cite[Section~4]{NSY1997}. Combining this last estimate with \eqref{a1b} entails that $v\in L_\infty((0,\infty)\times\Omega)$ and completes the proof.
\end{proof}

\begin{proof}[Proof of Theorem~\ref{thma2}~(b)]
Since $w$ is radially symmetric, it satisfies an improved version of Proposition~\ref{propd2}: for all $\eta>0$, there is $K(\eta)>0$ depending only on $\eta$ such that
\begin{equation*}
\int_\Omega e^w \ \mathrm{d}x \le K(\eta) \exp{\left[ \left( \eta+ \frac{1}{16\pi} \right) \|\nabla w\|_2^2 + \frac{2}{|\Omega|} \|w\|_1 \right] }\ ,
\end{equation*}
see \cite[Theorem~2.1]{NSY1997}. Consequently, 
\begin{equation*}
- M \ln{\| e^w \|_1} \ge - M \ln{K(\eta)} - M \left( \eta+ \frac{1}{16\pi} \right) \|\nabla w\|_2^2 - \frac{2M}{|\Omega|} \|w\|_1\ ,
\end{equation*}
and it follows from \eqref{d3} and \eqref{d7a} that
\begin{equation*}
\mathcal{E}_0(u,w) \ge \frac{8\pi D - (1+16\pi\eta)M}{16\pi} \|\nabla w\|_2^2 + \frac{\delta}{2} \|w\|_2^2 - b - M \ln{K(\eta)} \ .
\end{equation*}
Choosing $\eta = (8\pi D - M)/(32\pi M)>0$ in the previous inequality gives
\begin{equation*}
\mathcal{E}_0(u,w) \ge \frac{8\pi D - M}{32\pi} \|\nabla w\|_2^2 + \frac{\delta}{2} \|w\|_2^2 - b \ .
\end{equation*}
We then proceed as in Section~\ref{s3.2} to complete the proof.
\end{proof}

%%%%%%%%%%%%%%%%
%%%%%%%%%%%%%%%%
\subsection{Stabilization}\label{s3.3}
%%%%%%%%%%%%%%%%
%%%%%%%%%%%%%%%%

Another useful consequence of the availability of a Liapunov functional is the stabilization of bounded solutions; that is, any cluster point as $t\to\infty$ of a bounded solution to \eqref{a1}-\eqref{a2} is a stationary solution to \eqref{a1}. Not only does this information provide better insight into the dynamics of solutions to \eqref{a1}-\eqref{a2} when $\|u^{in}\|_1<4\pi D$, since solutions are known to be bounded in that case by Theorem~\ref{thma2}, but it is also an important step in the construction of unbounded solutions to be performed in Section~\ref{s4}. This fact has already been observed in \cite{GZ1998} and used for similar purposes in \cite{Ho2001, Ho2002, HW2001}.

%%%%%%%%%%%%%%%%
\begin{proposition}
Let $\theta\in (5/6,1)$ and consider initial conditions $(u^{in},v^{in},w^{in})\in \mathcal{I}_{M,\theta}$. We denote the corresponding solution to \eqref{a1}-\eqref{a2} given by Theorem~\ref{thma1} by $(u,v,w)$ and assume further that there is $\Lambda>0$ such that
\begin{equation}
\|u(t)\|_\infty \le \Lambda\ , \qquad t\ge 0\ . \label{e1}
\end{equation}
There are a sequence $(t_k)_{k\ge 1}$, $t_k\to\infty$, of positive real numbers and non-negative functions $(u_*,v_*,w_*)\in L_\infty(\Omega;\mathbb{R}^2)\times W_{2,\mathcal{B}}^2(\Omega)$ such that 
\begin{align}
& \lim_{k\to\infty} \left[ \|u(t_k)-u_*\|_2 + \|v(t_k)-v_*\|_2 + \|w(t_k)-w_*\|_{W_2^1} \right] = 0\ , \label{e3b} \\
& \lim_{k\to \infty} \mathcal{E}_0(u(t_k),w(t_k)) = \mathcal{E}_0(u_*,w_*)\ , \label{e3a} \\
& \mathcal{E}(u_*,v_*,w_*) \le \liminf_{k\to\infty} \mathcal{E}(u(t_k),v(t_k),w(t_k))\ , \label{e3e} 
\end{align}
where
\begin{equation}
u_* = v_* = M \frac{e^{w_*}}{\|e^{w_*}\|_1} \label{e3c}\ ,
\end{equation}
and $w_*$ solves the (nonlocal) elliptic equation
\begin{equation}
-D \Delta w_* + \delta w_* = M \frac{e^{w_*}}{\|e^{w_*}\|_1} \;\;\text{ in }\;\; \Omega\ , \qquad \nabla w_*\cdot \mathbf{n} = 0 \;\;\text{ on }\;\; \partial\Omega\ . \label{e3d}
\end{equation}
In other words, $(u_*,v_*,w_*)$ is a stationary solution to \eqref{a1}. 
\end{proposition}
%%%%%%%%%%%%%%%%

\begin{proof}
A first consequence of \eqref{d3} and \eqref{e1} is that, for $t\ge 0$, 
\begin{equation*}
\int_\Omega u(t,x) w(t,x)\ \mathrm{d}x \le \|u(t)\|_\infty \|w(t)\|_1 \le b \Lambda\ ,
\end{equation*}
so that
\begin{equation*}
\mathcal{E}_0(u(t),w(t)) \ge \|L(u(t))\|_1 + \frac{D}{2} \|\nabla w(t)\|_2^2 + \frac{\delta}{2} \|w(t)\|_2^2 - b \Lambda
\end{equation*}
by \eqref{a3a}. Together with \eqref{a1c} and Lemma~\ref{lemd2}, the previous inequality implies that
\begin{equation}
\|L(u(t))\|_1 + \|w(t)\|_{W_2^1} + \|\partial_t w(t)\|_2^2 + \int_0^t \mathcal{D}(u(s),v(s),w(s))\ \mathrm{d}s \le b(1+\Lambda) \label{e2}
\end{equation}
for $t\ge 0$. It next follows from \eqref{a1b} and \eqref{e1} that
\begin{equation}
\|v(t)\|_\infty + \|\partial_t v(t)\|_\infty \le \left( 1 + \frac{1}{\nu\varepsilon} \right)\left( \|v^{in}\|_\infty + \Lambda \right)\ , \qquad t\ge 0\ . \label{e4}
\end{equation}
Also, \eqref{a1c}, \eqref{e2}, and \eqref{e4} entail that
\begin{equation*}
D \|\Delta w(t)\|_2 \le \nu \|\partial_t w(t)\|_2 + \delta \|w(t)\|_2 + \|v(t)\|_2 \le b(1+\Lambda)\ , \qquad t\ge 0\ ,
\end{equation*}
from which we deduce that
\begin{equation}
\|w(t)\|_{W_2^2} \le b(1 + \Lambda)\ , \qquad t\ge 0\ . \label{e5}
\end{equation}
The remainder of the proof is then the same as that of \cite[Theorem~5.2]{GZ1998}, see also \cite[Lemma~1]{Ho2002}. The only additional point is to derive \eqref{e3e}. To this end, we notice that \eqref{e3b} and \eqref{e5} entail that $(-D \Delta w(t_k) + \delta w(t_k) - v(t_k))_{k\ge 1}$ converges weakly towards $-D\Delta w_* + \delta w_* - v_*$ in $L_2(\Omega)$, which implies that
\begin{equation*}
\| -D\Delta w_* + \delta w_* - v_* \|_2 \le \liminf_{k\to \infty} \|-D \Delta w(t_k) + \delta w(t_k) - v(t_k)\|_2
\end{equation*}
by a weak lower semicontinuity argument. Combining this property with \eqref{e3a} gives \eqref{e3e} .
\end{proof}

%%%%%%%%%%%%%%%%
%%%%%%%%%%%%%%%%
\section{Unbounded solutions}\label{s4}
%%%%%%%%%%%%%%%%
%%%%%%%%%%%%%%%%

This section is devoted to the proof of Theorem~\ref{thma3} which is somewhat a direct consequence of the analysis performed in \cite{Ho2001, Ho2002, HW2001}. The main step is actually to recast the problem in a way suitable to apply results from \cite{HW2001}.

Given $M>0$, we define the set $\mathcal{S}_M$ as follows: $(u,v,w) \in \mathcal{S}_M$ if 
\begin{subequations}\label{f0}
\begin{align}
& (u,v,w)\in W_{2,\mathcal{B}}^2(\Omega;\mathbb{R}^3)\ , \label{f0a} \\
& u = v = M \frac{e^w}{\|e^w\|_1} \ge 0 \;\text{ in }\; \Omega\ , \qquad  w\ge 0 \;\text{ in }\; \Omega\ , \label{f0b} \\
& - D \Delta w + \delta w = M \frac{e^w}{\|e^w\|_1} \;\text{in }\; \Omega\ , \qquad \nabla w \cdot \mathbf{n} = 0 \;\text{ on }\; \partial\Omega\ . \label{f0c}
\end{align}
\end{subequations}
Equivalently, $\mathcal{S}_M$ is the set of non-negative stationary solutions $(u,v,w)$ to \eqref{a1} which belong to $W_{2,\mathcal{B}}^2(\Omega;\mathbb{R}^3)$ and for which $\|u\|_1 = M$. We begin with a lower bound of the Liapunov functional $\mathcal{E}$ (defined in \eqref{a3}) on $\mathcal{S}_M$ for suitable values of $M$. 

%%%%%%%%%%%%%%%%
\begin{proposition}\label{propf1}
\begin{itemize}
	\item[(a)] If $M\in (4\pi D,\infty)\setminus 4\pi D \mathbb{N}$, then
	\begin{equation*}
	\mu_M := \inf_{(u,v,w)\in \mathcal{S}_M} \mathcal{E}(u,v,w) > -\infty\ .
	\end{equation*}
	\item[(b)] If $\Omega=B_r(0)$ for some $r>0$ and $M\in (8\pi D,\infty)$, then
	\begin{equation*}
	\mu_M := \inf_{(u,v,w)\in \mathcal{S}_{M,\text{rad}}} \mathcal{E}(u,v,w) > -\infty\ ,
	\end{equation*}
	where $\mathcal{S}_{M,\text{rad}} := \{ (u,v,w)\in \mathcal{S}_M\ :\ u,v,w \;\text{ are radially symmetric} \}$.
\end{itemize}
\end{proposition}
%%%%%%%%%%%%%%%%

\begin{proof}  -- Case~(a). Consider $(u,v,w)\in \mathcal{S}_M$. Introducing 
\begin{equation*}
E(w) := \frac{D}{2} \|\nabla w\|_2^2 + \frac{\delta}{2} \|w\|_2^2 - M \ln{\|e^w\|_1} + M\ln{M} - M + |\Omega|\ ,
\end{equation*}
it follows from \eqref{d7a} that
\begin{equation}
\mathcal{E}(u,v,w) \ge \mathcal{E}_0(u,w) \ge E(w)\ , \label{f1}
\end{equation}
while the fact that $(u,v,w)\in\mathcal{S}_M$ implies that
\begin{equation}
w \ge 0 \;\text{ in }\; \Omega \;\;\text{ and }\;\; \|w\|_1 = \frac{M}{\delta}\ . \label{f2}
\end{equation}
Setting 
\begin{equation*}
W := w - \frac{M}{\delta |\Omega|} = w - \frac{1}{|\Omega|} \int_\Omega w(x)\ \mathrm{d}x\ ,
\end{equation*}
an alternative formula for $E(w)$ in terms of $W$ reads
\begin{align}
E(w) & = \frac{D}{2} \|\nabla W\|_2^2 + \frac{\delta}{2} \|W\|_2^2 - M \ln{\|e^W\|_1} - \frac{M^2}{2\delta |\Omega|} + M\ln{M} - M + |\Omega| \nonumber\\
& = \frac{M}{|\Omega|} \mathcal{F}(W) - \frac{M^2}{2\delta |\Omega|} + M\ln{M} - M + |\Omega| - M \ln{|\Omega|}\ , \label{f3}
\end{align}
with
\begin{equation}
\mathcal{F}(W) := \frac{|\Omega|}{2M} \left( D \|\nabla W\|_2^2 + \delta \|W\|_2^2 \right) - |\Omega| \ln{\left( \frac{\|e^W \|_1}{|\Omega|} \right)}\ . \label{f10}
\end{equation}
Furthermore, since $w$ solves \eqref{f0c}, $W$ is a solution to
\begin{equation*}
- D \Delta W + \delta W = \frac{M}{|\Omega|} \left( \frac{|\Omega| e^W}{\|e^W\|_1} - 1 \right) \;\text{in }\; \Omega\ , \qquad \nabla W \cdot \mathbf{n} = 0 \;\text{ on }\; \partial\Omega\ .
\end{equation*}
Since $M\in (4\pi D,\infty)\setminus 4\pi D \mathbb{N}$, we infer from \cite[Lemma~3.5]{HW2001} that there is $\mu\ge 0$ which does not depend on $W$ such that
\begin{equation}
\mathcal{F}(W) \ge - \mu\ . \label{f4}
\end{equation}
Combining \eqref{f1}, \eqref{f3}, and \eqref{f4} completes the proof. 

\medskip

\noindent -- Case~(b). The proof is the same as that for case~(a), except that we use \cite[Corollary~3.7 \& Remark~3.8]{HW2001} instead of \cite[Lemma~3.5]{HW2001} to obtain the lower bound \eqref{f4} for all $M>8\pi D$.
\end{proof}

The second step is to show that $\mathcal{E}$ is not bounded from below on the set $\mathcal{I}_{M,\theta}$ of initial data to which Theorem~\ref{thma1} applies. 

%%%%%%%%%%%%%%%%
\begin{proposition}\label{propf2}
	Let $\theta\in (5/6,1)$. 
\begin{itemize}
	\item[(a)] If $M>4\pi D$, then 
	\begin{equation*}
	\inf_{(u,v,w)\in \mathcal{I}_{M,\theta}} \mathcal{E}(u,v,w) = - \infty\ .
	\end{equation*}
	\item[(b)] If $\Omega=B_r(0)$ for some $r>0$ and $M>8\pi D$, then 
	\begin{equation*}
	\inf_{(u,v,w)\in \mathcal{I}_{M,\theta}^{\text{rad}}} \mathcal{E}(u,v,w) = - \infty\ ,
	\end{equation*}
	where $\mathcal{I}_{M,\theta}^{\text{rad}} := \{ (u,v,w)\in \mathcal{I}_{M,\theta}\ :\ u,v,w \;\text{ are radially symmetric} \}$.
\end{itemize}
\end{proposition}
%%%%%%%%%%%%%%%%

\begin{proof} -- Case~(a). Fix $x_0\in\partial\Omega$. For $\eta\in (0,\infty)$ and $x\in\Omega$, define
\begin{equation*}
\Theta_\eta(x) := 2 \ln{\left( \frac{\eta}{(\eta^2+\pi |x-x_0|^2)} \right)} - \frac{2}{|\Omega|} \int_\Omega \ln{\left( \frac{\eta}{(\eta^2+\pi |y-x_0|^2)} \right)}\ \mathrm{d}y\ .
\end{equation*}
Then $\Theta_\eta\in W_2^1(\Omega)$ with $\langle \Theta_\eta \rangle=0$ and satisfies
\begin{equation}
\lim_{\eta\to 0} \mathcal{F}(\Theta_\eta) = - \infty \;\text{ and } \lim_{\eta\to 0} \|\nabla \Theta_\eta\|_2 = \infty \label{f5}
\end{equation}
by \cite[Section~3]{HW2001}, since $M>4\pi D$ and the boundary of $\Omega$ is smooth. Let $\eta\in (0,\infty)$. Since $W_3^1(\Omega) \stackrel{\cdot}{=} [L_3(\Omega),W_{3,\mathcal{B}}^2(\Omega)]_{1/2}$ by \cite[Theorem~5.2]{Am1993} and $W_3^1(\Omega)$ is densely embedded in $W_2^1(\Omega)$, there is a sequence $(W_{j,\eta})_{j\ge 1}$ in $W_{3,\mathcal{B}}^2(\Omega)$ such that
\begin{equation}
\|W_{j,\eta}\|_{W_2^1}\le 2 \|\Theta_\eta\|_{W_2^1}\ , \quad \langle W_{j,\eta} \rangle = 0\ , \quad j\ge 1 \;\text{ and }\; \lim_{j\to\infty} \|W_{j,\eta} - \Theta_\eta\|_{W_2^1} = 0\ . \label{f6}
\end{equation}
In addition, possibly after extracting a subsequence, we may assume that
\begin{equation}
\lim_{j\to\infty} W_{j,\eta}(x) = \Theta_\eta(x) \;\text{ for a.e. }\; x\in\Omega\ . \label{f8}
\end{equation}
In addition, it follows from H\"older's inequality and Proposition~\ref{propd2} that, for any measurable subset $\omega$ of $\Omega$,
\begin{align*}
\int_\omega e^{W_{j,\eta}}\ \mathrm{d}x & \le |\omega|^{1/2} \left( \int_\omega e^{2W_{j,\eta}}\ \mathrm{d}x\right)^{1/2} \\
& \le \sqrt{\kappa_0} |\omega|^{1/2} \exp{\left( \frac{\|\nabla W_{j,\eta}\|_2^2}{4\pi} + \frac{\|W_{j,\eta}\|_1}{|\Omega|} \right)} \\
& \le  \sqrt{\kappa_0} |\omega|^{1/2} \exp{\left( \frac{\|\Theta_{\eta}\|_{W_2^1}^2}{\pi} + \frac{\sqrt{2} \|\Theta_{\eta}\|_{W_2^1}}{\sqrt{|\Omega|}} \right)}\ ,
\end{align*}
which allows us to deduce from Dunford-Pettis' theorem that
\begin{equation*}
\left( e^{W_{j,\eta}} \right)_{j\ge 1} \;\text{ is relatively sequentially weakly compact in }\; L_1(\Omega)\ . 
\end{equation*}
Since $\left( e^{W_{j,\eta}} \right)_{j\ge 1}$ converges to $e^{\Theta_\eta}$ a.e. in $\Omega$ by \eqref{f8}, we deduce from these two properties and Vitali's theorem that
\begin{equation}
\lim_{j\to\infty} \left\| e^{W_{j,\eta}} - e^{\Theta_\eta} \right\|_1 = 0\ . \label{f9}
\end{equation}
Recalling the definition \eqref{f10} of $\mathcal{F}$, we infer from \eqref{f6} and \eqref{f9} that
\begin{equation}
\lim_{j\to\infty} \mathcal{F}(W_{\eta,j}) = \mathcal{F}(\Theta_\eta)\ . \label{f11}
\end{equation}
Owing to \eqref{f6} and \eqref{f11}, there is $j_\eta\ge 1$ such that the function $W_\eta := W_{j_\eta,\eta}$ satisfies
\begin{equation}
W_\eta\in W_{3,\mathcal{B}}^2(\Omega) \;\text{ and }\; \left\|W_{\eta} - \Theta_\eta \right\|_{W_2^1} + \left| \mathcal{F}(W_{\eta}) - \mathcal{F}(\Theta_\eta) \right| \le \eta\ . \label{f12}
\end{equation}

After this preparation, for $\eta>0$, we set
\begin{equation}
w_\eta := W_{\eta} + \frac{M}{\delta |\Omega|}\ , \quad u_\eta := M \frac{e^{w_\eta}}{\|e^{w_\eta}\|_1}\ , \label{f13}
\end{equation}
and, since $-D\Delta w_\eta + \delta w_\eta\in L_3(\Omega)$, we fix 
\begin{equation}
v_\eta\in W_3^1(\Omega) \;\;\text{ such that }\;\; \| v_\eta + D \Delta w_\eta - \delta w_\eta \|_3 \le \eta\ .  \label{f14}
\end{equation}
Now, owing to \eqref{a3a}, \eqref{f3}, and \eqref{f13},
\begin{align*}
\mathcal{E}_0(u_\eta,w_\eta) & = M\ln{M} - M \ln{\|e^{w_\eta}\|_1} - M + |\Omega| + \frac{D}{2} \|\nabla w_\eta\|_2^2 + \frac{\delta}{2} \|w_\eta\|_2^2 \\
& = E(W_\eta) =  \frac{M}{|\Omega|} \mathcal{F}(W_\eta) + \chi\ ,
\end{align*}
with
\begin{equation*}
\chi := - \frac{M^2}{2\delta |\Omega|} + M\ln{M} - M + |\Omega| - M \ln{|\Omega|}\ , 
\end{equation*}
so that
\begin{equation}
\mathcal{E}(u_\eta,v_\eta,w_\eta) = \frac{M}{|\Omega|} \mathcal{F}(W_\eta) + \frac{\varepsilon}{2} \|-D\Delta w_\eta + \delta w_\eta - v_\eta \|_2^2 + \chi\ . \label{f15}
\end{equation}
We then infer from \eqref{f12}, \eqref{f14}, \eqref{f15}, and H\"older's inequality that
\begin{align*}
\mathcal{E}(u_\eta,v_\eta,w_\eta) & \le \frac{M}{|\Omega|} \mathcal{F}(\Theta_\eta) + \frac{\eta M}{|\Omega|} + \frac{\varepsilon |\Omega|^{1/3}}{2} \|-D \Delta w_\eta + \delta w_\eta - v_\eta\|_3^2 + \chi\\
& \le \frac{M}{|\Omega|} \mathcal{F}(\Theta_\eta) + \frac{\eta M}{|\Omega|} + \frac{\varepsilon \eta^2 |\Omega|^{1/3}}{2} + \chi \ .
\end{align*}
Taking the limit $\eta\to 0$ in the previous inequality and using \eqref{f5} lead us to
\begin{equation*}
\lim_{\eta\to 0} \mathcal{E}(u_\eta,v_\eta,w_\eta) = - \infty\ , 
\end{equation*}
and completes the proof. 

\medskip

\noindent -- Case~(b). The proof is the same as the previous one except that the starting point is \cite[Lemma~2]{Ho2002}, the function $\Theta_\eta$ being defined as in Case~(a) but with $x_0=0$.
\end{proof}
	
\begin{proof}[Proof of Theorem~\ref{thma3}] -- Case~(a). Consider $M\in (4\pi D,\infty)\setminus 4\pi D \mathbb{N}$ and $\theta\in (5/6,1)$. According to Proposition~\ref{propf2}, there is $(u^{in},v^{in},w^{in})\in \mathcal{I}_{M,\theta}$ such that  
\begin{equation}
\mathcal{E}(u^{in},v^{in},w^{in}) < \mu_M\ , \label{f16}
\end{equation}
the parameter $\mu_M$ being defined in Proposition~\ref{propf1}. Assume for contradiction that there is $\Lambda>0$ such that $\|u(t)\|_\infty\le \Lambda$ for all $t\ge 0$. We are then in a position to apply Proposition~\ref{propd2}~(a) and deduce that there are a sequence $(t_k)_{k\ge 1}$, $t_k\to\infty$, of positive real numbers and $(u_*,v_*,w_*)\in \mathcal{S}_M$ such that
\begin{equation*}
\mathcal{E}(u_*,v_*,w_*) \le \liminf_{k\to\infty} \mathcal{E}(u(t_k),v(t_k),w(t_k)) \ .
\end{equation*} 
Since $\mathcal{E}$ is a Liapunov function for \eqref{a1} by Lemma~\ref{lemd3} and $(u_*,v_*,w_*)\in \mathcal{S}_M$, we infer from Proposition~\ref{propf1}~(a) and the previous identity that
\begin{equation*}
\mu_M \le \mathcal{E}(u_*,v_*,w_*) \le \liminf_{k\to\infty} \mathcal{E}(u(t_k),v(t_k),w(t_k)) \le \mathcal{E}(u^{in},v^{in},w^{in})\ ,
\end{equation*} 
which contradicts \eqref{f16} and completes the proof.

\medskip

\noindent -- Case~(b). The proof is the same as the previous one and relies on Proposition~\ref{propf1}~(b) and Proposition~\ref{propf2}~(b). 
\end{proof}

%%%%%%%%%%%%%%%%
%%%%%%%%%%%%%%%%
\section{Related models}\label{s5}
%%%%%%%%%%%%%%%%
%%%%%%%%%%%%%%%%

To finish with, we just point out in an informal way that the related models \eqref{tw} and \eqref{twd} also have a Liapunov functional which is closely related to that of \eqref{a1}. Specifically, the functional $\mathcal{E}_0$ defined by \eqref{a3a} is a Liapunov functional for \eqref{tw} and, if $(u,v,w)$ is a solution to \eqref{tw}, then
\begin{equation*}
\frac{\mathrm{d}}{\mathrm{d}t} \mathcal{E}_0(u,w) = - \int_\Omega u |\nabla (\ln{u} - w)|^2\ \mathrm{d}x - D \|\nabla\partial_t w\|_2^2\ .
\end{equation*}
In the same vein, $\mathcal{E}_0$ is also a Liapunov functional for \eqref{twd} and, if $(u,v,w)$ is a solution to \eqref{twd}, then
\begin{equation*}
\frac{\mathrm{d}}{\mathrm{d}t} \mathcal{E}_0(u,w) = - \int_\Omega u |\nabla (\ln{u} - w)|^2\ \mathrm{d}x - \nu\varepsilon \left( D \|\nabla\partial_t w\|_2^2 + \delta \|\partial_t w\|_2^2 \right)\ .
\end{equation*}
It is then likely that the analysis performed for \eqref{a1} adapts to \eqref{tw} and \eqref{twd}, possibly with slightly different statements for \eqref{tw}. Indeed, Proposition~\ref{propf1}~(b) is only valid for $M\in (8\pi D,\infty)\setminus 8\pi D \mathbb{N}$ in that case, see \cite[Lemma~3]{Ho2002}.

%%%%%%%%%%%%%%%%
%%%%%%%%%%%%%%%%
\section*{Acknowledgments}
%%%%%%%%%%%%%%%%
%%%%%%%%%%%%%%%%

I gratefully thank Christoph Walker for helpful discussions during the preparation of this paper. Part of this work was done while enjoying the hospitality and support of the Institute of Mathematics of the Polish Academy of Sciences in Warsaw. 

%%%%%%%%%%%%%%%%
%%%%%%%%%%%%%%%%
\bibliographystyle{siam}
\bibliography{gbuscsisp}
%%%%%%%%%%%%%%%%
%%%%%%%%%%%%%%%%

\end{document}